\documentclass[11pt]{article}

\usepackage[utf8]{inputenc}
\usepackage{graphicx}
\usepackage{amsmath}
\usepackage{amssymb}
\usepackage{amsfonts}
\usepackage{amsthm}
\usepackage[margin=1.2in]{geometry}
\usepackage[font=small,width=\linewidth]{caption}
\usepackage[hidelinks]{hyperref}
\usepackage[title]{appendix}

\usepackage[linesnumbered, algoruled, vlined, algo2e]{algorithm2e}
\SetKwInput{KwInput}{Input}
\SetKwInput{KwOutput}{Output}
\SetKwBlock{Solve}{Solve}{end}
\SetKwBlock{Estimate}{Estimate}{end}
\SetKwBlock{Mark}{Mark}{end}
\SetKwBlock{Refine}{Refine}{end}
\usepackage{stmaryrd}%

\usepackage{tikz}
\usetikzlibrary{calc}
\usetikzlibrary{arrows.meta}

\numberwithin{equation}{section}

\theoremstyle{plain}

\newtheorem{theorem}{Theorem}[section]
\newtheorem{lemma}[theorem]{Lemma}
\newtheorem{corollary}[theorem]{Corollary}
\newtheorem{remark}[theorem]{Remark}

\usepackage{mathabx}
\usepackage{enumitem}
\usepackage{booktabs}

\usepackage[
    backend=biber,
    natbib=true,
    style=numeric-comp,
    sorting=none,
    maxnames=2,
    giveninits=true,
    url=false,
    arxiv=pdf,
    doi=false,
    eprint=true,
    isbn=false,
    date=terse,
    block=ragged,
    maxbibnames=99
]{biblatex}
\DeclareMathOperator{\grad}{\nabla}
\DeclareMathOperator{\ddiv}{div}
\DeclareMathOperator{\sep}{:}  %

\newcommand*{\dd}{\ensuremath{\,\mathrm{d}}}
\newcommand{\discretized}[1]{\ensuremath{\boldsymbol{#1}}}

\newenvironment{keywords}{%
    \begingroup%
    \small%
    \textbf{Key words.}}{\endgroup}
\newenvironment{AMS}{%
    \begingroup%
    \small%
    \textbf{AMS subject classifications.}}{\endgroup}

\title{On the convergence of adaptive Galerkin FEM for parametric PDEs with lognormal coefficients}
\author{
    Martin Eigel\footnote{Weierstrass Institute for Applied Analysis and Stochastics, Mohrenstr. 39, 10117 Berlin, Germany} \ and
    Nando Hegemann\footnote{Physikalisch-Technische Bundesanstalt, Abbestr. 2--12, 10587 Berlin, Germany} $^{,}$\footnote{corresponding author: nando.hegemann@ptb.de}
}
\date{\today}

\begin{document}

\maketitle

\begin{abstract}
    Numerically solving high-dimensional random parametric PDEs poses a challenging computational problem.
    It is well-known that numerical methods can greatly benefit from adaptive refinement algorithms, in particular when functional approximations in polynomials are computed as in stochastic Galerkin finite element methods.
    This work investigates a residual based adaptive algorithm, akin to classical adaptive FEM, used to approximate the solution of the stationary diffusion equation with lognormal coefficients, i.e. with a non-affine parameter dependence of the data.
    It is known that the refinement procedure is reliable but the theoretical convergence of the scheme for this class of unbounded coefficients remains a challenging open question.
    This paper advances the theoretical state-of-the-art by providing a quasi-error reduction result for the adaptive solution of the lognormal stationary diffusion problem.
    The presented analysis generalizes previous results in that guaranteed convergence for uniformly bounded coefficients follows directly as a corollary.
    Moreover, it highlights the fundamental challenges with unbounded coefficients that cannot be overcome with common techniques.
    A computational benchmark example illustrates the main theoretical statement.
\end{abstract}



\section{Introduction}%
\label{sec:intro}

In the natural sciences and engineering most modern simulation methods rely on \emph{partial differential equations} (PDEs).
In these applications, the simulation typically requires knowledge about many, only indirectly observable model data such as material properties or experimental inaccuracies.
Incorporating uncertainties or variations of the unknown parameters into the physical model easily leads to an extremely challenging discretization complexity, also known as the ``curse of dimensionality''.
To mitigate the numerical obstacles and obtain a better understanding of the problems underlying structure, model order reduction techniques such as RBM (reduced basis methods) and modern machine learning compression (neural networks and tensor networks) have been an essential area of research activity in the last decade.

One of the central notions in this field concerns constructing a solution iteratively, increasing the complexity locally (with respect to the different discretizations) only where it is necessary.
The main contribution of this work is to investigate and prove a guaranteed reduction of the quasi-error, consisting of error and error estimator, of such an adaptive algorithm, driven by a residual based error estimator, when applied to a certain class of elliptic parametric PDEs with unbounded\footnote{we use ``unbounded'' synonymously for ``not bounded uniformly''} coefficients.
This is analog to convergence results of classical adaptive finite element (FE) methods, which are foundational to this work.
As a model problem, we consider the parameter dependent stationary diffusion equation
\begin{equation}
    \begin{aligned}
      \label{eq:introduction:darcy}
      -\ddiv_x \bigl(a(x,y) \grad_x u(x,y)\bigr) &= f(x) \quad\mbox{in }D,\\
      u(x,y) &~= 0 \qquad\mbox{on }\partial D,
    \end{aligned}
\end{equation}
on a domain $D\subset\mathbb{R}^2$.
Here $y=(y_1,\dots,y_{\hat M})\subset\mathbb R^{\hat M}$ is a high or even infinite dimensional parameter vector determining the diffusion coefficient field $a$ and hence the solution $u$ of~\eqref{eq:introduction:darcy}.
In particular, we consider the coefficient to be the exponential of an affine random field $\gamma$ of Gaussian variables, i.e., $a(x,y) = \exp \gamma(x,y)$, which is generally referred to as a lognormal coefficient.
Several different adaptive schemes for both bounded (affine) and unbounded diffusion coefficients have been investigated in the literature~\cite{EigelGittelson2014asgfem,Bespalov2016,Bespalov2018,Crowder2019,Bespalov2021twolevel,Bachmayr2021adaptive,EigelMarschall2020lognormal,EigelFarchmin2022avmc}.
So far establishing convergence of the adaptive algorithms has only been accomplished for bounded affine diffusion coefficients~\cite{EigelGittelson2014convergence,BespalovPraetorius2019convergence,BespalovPraetorius2021optimality,EigelSprungk2022,bachmayr2024convergent}.

The main result of this work, Theorem~\ref{thm:convergence:convergence}, establishes that the proposed adaptive Algorithm~\ref{alg:estimator:adaptive}, which is based on the one described in~\cite{EigelFarchmin2022avmc}, reduces the quasi-error $\operatorname{err}_{\ell}^2 = \Vert \grad(u-u_\ell)\Vert^2 + \omega_\ell \eta_{\mathrm{det},\ell}^2 + \omega_\ell\tau \eta_{\mathrm{sto},\ell}^2$,\ $\omega_\ell,\tau>0$, in each iteration $\ell\in\mathbb{N}$, i.e.,
\begin{equation}%
  \label{eq:introduction:convergence}
  \begin{split}
    \operatorname{err}_{\ell+1}^2 \leq \delta_\ell \operatorname{err}_\ell^2,
    \qquad 0<\delta_\ell<1,
  \end{split}
\end{equation}
even for an unbounded lognormal coefficient $a$.
In our case the residual error estimator is composed of two contributions, namely $\eta_{\mathrm{det}}$ and $\eta_{\mathrm{sto}}$, and satisfies the reliability estimate $\Vert u-u_\ell\Vert^2 \lesssim \eta_\mathrm{det}(u_\ell)^2 + \eta_\mathrm{sto}(u_\ell)^2$ as in~\cite{EigelGittelson2014asgfem,EigelMarschall2020lognormal}.
We follow the proof in~\cite{EigelGittelson2014convergence}, which extends the basic strategies for deterministic finite element methods~\cite{Binev2004,Cascon2008,Doerfler1996,Morin2000,Stevenson2008} to the parametric setting.
Here, the major difference to the bounded affine case manifests in the adapted solution spaces necessary to guarantee well-posedness of~\eqref{eq:introduction:darcy}.
As a consequence, establishing basic properties such as Lipschitz continuity of $\eta_{\mathrm{det}}$ and $\eta_{\mathrm{sto}}$ or the additivity of $\eta_{\mathrm{sto}}$ with respect to the stochastic index set becomes non-trivial.

\begin{figure}[htpb]
  \centering
  \input{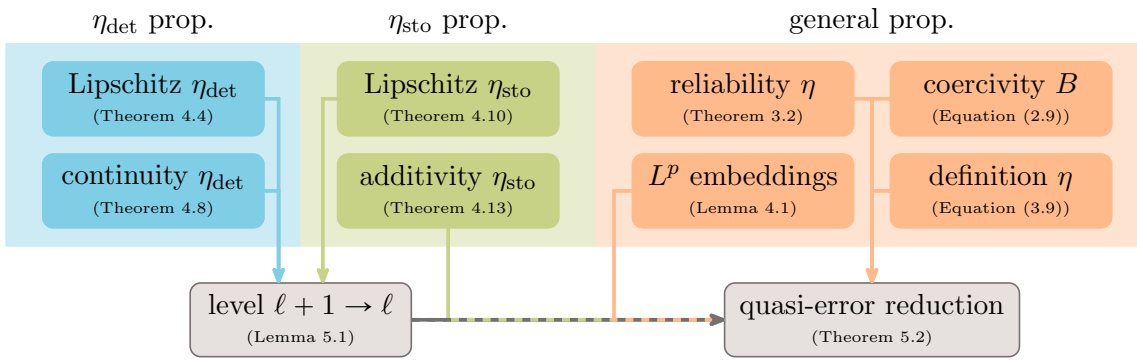}
  \caption{Schematic overview of the ingredients required for the quasi-error reduction proof and how they are employed.}%
  \label{fig:convergence:schematic}
\end{figure}

Figure~\ref{fig:convergence:schematic} visualizes schematically how the newly established properties are employed to prove that Algorithm~\ref{alg:estimator:adaptive} reduces the quasi-error~\eqref{eq:introduction:convergence} in each iteration.
We first utilize the Lipschitz continuity of $\eta_{\mathrm{det}}$ and $\eta_{\mathrm{sto}}$ as well as the continuity of $\eta_{\mathrm{det}}$ with respect to the index set (determining the active modes and hence the discrete parameter space) to prove Lemma~\ref{lem:convergence:main_lemma}, which yields an upper bound of the estimator contributions depending only on quantities of the previous iteration step.
Similarly, the embedding of weighted Gaussian $L^p$-spaces, the coercivity of the bilinear form of~\eqref{eq:introduction:darcy}in~\eqref{eq:setting:coercivity_B} and the reliability of the estimator lead to a bound of the error $\Vert u-u_{\ell+1}\Vert$ by quantities depending only on level $\ell$.
Refining either the spatial mesh or the stochastic discretization using a D\"orfler marking strategy allows us to derive a quasi-error reduction for the two different refinement scenarios.
Here, quasi additivity of $\eta_{\mathrm{sto}}$ in the index set is required to show the reduction in case $\eta_{\mathrm{sto}}$ is the dominating contribution.
Finally, the involved free constants have to be chosen appropriately to ensure~\eqref{eq:introduction:convergence} holds simultaneously for both spatial and stochastic refinement for some $0 < \delta_\ell < 1$.

We point out that while convergence with affine coefficients can be shown~\cite{EigelGittelson2014convergence,BespalovPraetorius2019convergence,Bachmayr2021adaptive}, the same strong statement cannot be achieved in our setting, despite the derived quasi-error reduction property in each refinement step.
This is an intrinsic consequence of the unboundedness of the diffusion coefficient.
In particular, the introduction of adapted function spaces to ensure well-posedness of~\eqref{eq:introduction:darcy} causes the degeneration of the Lipschitz constants of $\eta_{\mathrm{det}}$ and $\eta_{\mathrm{sto}}$ as the size of the stochastic index set increases.
Moreover, the energy norm induced by~\eqref{eq:introduction:darcy} and the canonical $L^2(\Gamma, \vartheta\rho; H_0^1(D))$ norm required to prove Lipschitz continuity are not equivalent~\cite{SchwabGittelson2011}, which allows $\delta_\ell \to 1$ as $\ell\to\infty$.
However, Theorem~\ref{thm:convergence:convergence} enables us to generalize the convergence proof of~\cite{EigelGittelson2014convergence} to arbitrary bounded coefficients and show convergence with arbitrarily high probability for a lognormal coefficient in practical applications.
The latter is confirmed by our numerical experiments.

The remainder of this work is structured as follows.
Section~\ref{sec:setting} introduces the model problem setting, its variational formulation as well as the spatial and stochastic discretization.
In Section~\ref{sec:estimator}, we define the residual based error estimator and introduce the algorithm that steers the adaptive refinement of the spatial mesh and the stochastic index set.
Thereafter, we show some basic properties of the error estimator contributions in Section~\ref{sec:properties}, which are essential to prove the main result in Section~\ref{sec:convergence}.
Finally, we show numerical evidence to confirm the theoretical results in Section~\ref{sec:experiments}.

\section{Parametric model problem}%
\label{sec:setting}

This section introduces the required notation for the parametric model problem~\eqref{eq:introduction:darcy}.
We give a short overview of the analytical setting and describe the results necessary to overcome most of the technical challenges caused by the unbounded lognormal coefficient.

\subsection{Stationary diffusion with lognormal coefficient}%
\label{sec:setting:darcy}

Let $D\subset \mathbb{R}^2$ be a polygonal bounded Lipschitz domain representing the spatial computational area.
Note that we restrict our setting to two physical dimensions only to avoid additional notational complexity, which would not add any relevant insights.
All results also hold true for $D\subset\mathbb{R}^d$ for $d=1,2,3$, see e.g.~\cite{EigelMarschall2020lognormal} for details.
For $\hat M\in\mathbb{N}\cup\{ \infty \}$ and almost all $y\in\mathbb{R}^{\hat M}$, let
\begin{equation}
  \label{eq:setting:gamma}
  \gamma(x,y) = \sum_{m=1}^{\hat M} \gamma_m(x) y_m
  \qquad\mbox{for any } x\in D.
\end{equation}
We define the diffusion coefficient $a$ in~\eqref{eq:introduction:darcy} as the exponential of $\gamma$ and note that the affine structure of $\gamma$ is a common representation in Uncertainty Quantification.
In a typical setting, where the randomness of the parametric problem is given through a random field with known covariance kernel, the Karhunen-Lo\`eve expansion~\cite{YingboHua1998} is a popular tool to decorrelate the random field into an affine sum similar to~\eqref{eq:setting:gamma}.
The unbounded diffusion coefficient is now given by
\begin{equation}
  \label{eq:setting:a}
  a(x,y) = \exp( \gamma(x,y)).
\end{equation}
Without loss of generality we restrict this work to a deterministic source term $f\in L^2(D)$ and homogeneous Dirichlet boundary conditions.

The unboundedness of coefficient $a$ yields an ill-posed problem and adapted function spaces have to be introduced, cf.~\cite{Bachmayr2017sparse,Hoang2014,Gittelson2010,Mugler2013,GalvisSarkis2009}.
In the following, we give a summary of the necessary concepts and refer to~\cite{EigelMarschall2020lognormal,SchwabGittelson2011} for a concise derivation of the analysis.
Let $\mathcal{X}:=H_0^1(D)$ be equipped with the norm $\Vert w\Vert_\mathcal{X}=\Vert \grad w\Vert_{L^2(D)}$, $\Vert w\Vert_D = \Vert w\Vert_{L^2(D)}$ and let $\mathcal{F} := \{ \mu\in\mathbb{N}_0^{\infty}\colon \vert \operatorname{supp}\mu\vert < \infty \}$ be the set of finitely supported multi-indices, where $\operatorname{supp}\mu$ denotes the set of all indices of $\mu$ different from zero.
The set of admissible parameters $y\in\mathbb{R}^{\hat M}$ is then given by
\begin{equation}
  \label{eq:setting:Gamma}
  \Gamma := \left\{ y\in\mathbb{R}^{\hat M}\colon \sum_{m=1}^{\hat M} \Vert \gamma_m\Vert_{L^{\infty}(D)} \vert y_m \vert < \infty \right\},
\end{equation}
which excludes only a set of measure zero from $\mathbb{R}^{\hat M}$~\cite{SchwabGittelson2011}.
For any $\rho\geq0$, we set $\sigma_m(\rho) = \exp(\rho\Vert \gamma_m \Vert_{L^\infty(D)}) \geq 1$ and define
\begin{equation}
  \label{eq:setting:zeta}
  \zeta_\rho(y) := \prod_{m=1}^{\hat M}\zeta_{\rho,m}(y_m)
  \quad\mbox{for}\quad
  \zeta_{\rho,m}(y_m) := \frac{1}{\sigma_m(\rho)} \exp\Bigl( \bigl(\frac{1}{2} -\frac{1}{2\sigma_m(\rho)^2}\bigr) y_m^2 \Bigr).
\end{equation}
Since $\zeta_{\rho,m}$ is the Radon-Nikodym derivative of $\mathcal{N}(0, \sigma_m(\rho)^2)$ with respect to the standard Gaussian measure, we can define the probability density function of $\mathcal{N}(0, \sigma_m(\rho)^2)$ via
\begin{equation}
  \label{eq:setting:pi}
  \pi_{\rho,m}(y_m) = \zeta_{\rho,m}(y_m) \frac{1}{\sqrt{2\pi}}\exp\left(-\frac{1}{2}y_m^2\right).
\end{equation}
Note that $\pi_{0,m}$ is the density of the standard Gaussian measure, since $\zeta_{0,m}\equiv 1$.
With this, we define the weighted product measure $\pi_{\rho}(y) := \prod_{m=1}^{\hat M} \pi_{\rho,m}(y_m)$.
Since we assume $y_m\sim\mathcal{N}(0,\sigma_m(\rho)^2)$ for some $\rho>0$, we refer to $a$ as a lognormal coefficient field.

Regarding the variational formulation, for any $\vartheta\in[0,1]$ and $\rho>0$, define the bilinear form for all $w,v\in L^2(\Gamma,\pi_{\vartheta\rho};\mathcal{X})$ by
\begin{equation}
  \label{eq:setting:B}
  B(w,v) := B_{\vartheta\rho}(w,v)
  := \int_{\Gamma} \int_{D} a(x,y)\grad w(x,y)\cdot\grad v(x,y) \,\mathrm{d}x \,\mathrm{d}\pi_{\vartheta\rho}(y)
\end{equation}
and denote by $\Vert w\Vert_{B} = B(w,w)^{1/2}$ the energy norm induced by $B$.
Additionally, let $\Vert w\Vert_{\pi_{\vartheta\rho},D}=\Vert w\Vert_{L^2(\Gamma,\pi_{\vartheta\rho};L^2(D))}$.
The variational form of~\eqref{eq:introduction:darcy} then reads: Find $u\in\mathcal{V}$ such that
\begin{equation}
  \label{eq:setting:varriational_form}
  B(u,v) = F(v)
  \quad\mbox{for all }v\in\mathcal{V},
\end{equation}%
where $F(v):=\int_{\Gamma}\int_D f(x)v(x,y)\,\mathrm{d}x\,\mathrm{d}\pi_{\vartheta\rho}(y)$ and
\begin{equation*}
  \mathcal{V} := \{ w\colon\Gamma\to\mathcal{X} \mbox{ measurable with }B_{\vartheta\rho}(w,w)<\infty \}.
\end{equation*}

\begin{remark}%
  With this problem adapted function space, it follows by~\cite[Proposition 2.43]{SchwabGittelson2011} that
  \begin{equation*}
    L^2(\Gamma,\pi_{\rho};\mathcal{X})
    \subset \mathcal{V}
    \subset L^2(\Gamma,\pi_{0};\mathcal{X})
  \end{equation*}
  are continuous embeddings for any $0<\vartheta<1$.
  Thus, the bilinear form $B_{\vartheta\rho}$ is $\mathcal{V}$--elliptic and bounded in the sense that
  \begin{align}
    \label{eq:setting:boundedness_B}
    \vert B_{\vartheta\rho}(w,v) \vert
    &\leq \hat{c}_{\vartheta\rho}\, \Vert \grad w\Vert_{\pi_{\rho},D}\, \Vert \grad v\Vert_{\pi_{\rho},D}
    &\mbox{for all }w,v\in L^2(\Gamma,\pi_{\rho};\mathcal{X}),\\
    \label{eq:setting:coercivity_B}
    B_{\vartheta\rho}(w,w) 
    &\geq \check{c}_{\vartheta\rho} \Vert \grad w\Vert_{\pi_{0},D}^2
    &\mbox{for all }w\in L^2(\Gamma,\pi_{0};\mathcal{X}),
  \end{align}
  for some $\hat{c}_{\vartheta\rho}, \check{c}_{\vartheta\rho} > 0$, which implies the well-posedness of~\eqref{eq:setting:varriational_form}.
\end{remark}

\subsection{Discretization}%
\label{sec:setting:discretization}

Since the spatial domain $D$ has a polygonal boundary, we assume a shape regular triangulation $\mathcal{T}$ that represents $D$ exactly.
We denote by $\mathcal{E}$ the set of edges of $\mathcal{T}$ and let $\partial T=\{ E\in\mathcal{E} \colon E\cap\bar{T}\neq\emptyset\}$ for any triangle $T\in\mathcal{T}$.
For any $E\in\mathcal{E}$ denote by $h_E=\vert E \vert$ the length of $E$ and let $h_T=\max_{E\in\partial T} h_E$ denote the diameter of any $T\in\mathcal{T}$.

For the spatial discretization, consider the conforming Lagrange finite element space of order $p$
\begin{equation}
  \label{eq:setting:X}
  \mathcal{X}_p(\mathcal{T})
  = P_p(\mathcal{T})\cap C(\bar{\mathcal{T}})
  =\operatorname{span}\{ \varphi_j \}_{j=1}^{J},
\end{equation}
where $P_p(\mathcal{T})$ denotes the space of element-wise polynomials of order $p$.
Here, $J\in\mathbb{N}$ denotes the dimension of the finite element space.
We define the normal jump of a function $\xi\in H^1(D;\mathbb{R}^2)$ over the edge $E=\bar{T}\cap \bar{T'}$ by $\llbracket \xi \rrbracket_E = (\xi|_{T} - \xi|_{T'})\cdot n_E$ for the edge normal vector $n_E=n_{T}=-n_{T'}$ of $E$.
Since the direction of the normal $n_E$ depends on the enumeration of the neighbouring triangles, we assume an arbitrary but fixed choice of the sign of $n_E$ for each $E\in\mathcal{E}$.
For any $w_N\in\mathcal{X}_p(\mathcal{T})$ and $\xi_N\in\grad\mathcal{X}_p(\mathcal{T}) = \{\grad w_N\colon w_N\in\mathcal{X}_p(\mathcal{T})\}$, recall the inverse estimates~\cite{Braess2007}
\begin{equation}
    \label{eq:setting:inverse_estimates}
    h_T^{1/2} \Vert \grad w_N\cdot n_T\Vert_{\partial T\cap D} \leq c_{\mathrm{inv}}\Vert \grad w_N \Vert_T
    \qquad\mbox{and}\qquad
    h_T\Vert \ddiv \xi_N\Vert_T \leq c_{\mathrm{inv}} \Vert \xi_N \Vert_T.
\end{equation}

For the discretization of the parameter space, by $\{P^m_k\}_{k=0}^\infty$ we denote a set of orthogonal and normalized polynomials in $L^2(\Gamma,\pi_{\vartheta\rho,m})$.
This defines a tensorized orthonormal product basis $\{ P_\mu \}_{\mu\in\mathcal{F}}$ of $L^2(\Gamma,\pi_{\vartheta\rho})$ by $P_\mu(y) := \prod_{m=1}^{\hat M} P_{\mu_m}^m(y_m) = \prod_{m\in\operatorname{supp}\mu} P_{\mu_m}^m(y_m)$.
Due to the weighted $L^2$ spaces of our functional setting, the basis $\{ P_\mu \}$ consists of scaled Hermite polynomials, see~\cite{SchwabGittelson2011,EigelFarchmin2022avmc}.
Note that the use of global polynomials is justified by the high (holomorphic) regularity of the solution of~\eqref{eq:introduction:darcy} with respect to the stochastic variables~\cite{Cohen2010,Cohen2011,Hoang2014}.
For any $\alpha,\beta,\mu\in\mathbb{N}_0^{\hat M}$ we define the triple product
\begin{equation}%
  \label{eq:setting:tau}
  \tau_{\alpha\beta\mu}:=\prod_{m=1}^{\hat M} \tau_{\alpha_m\beta_m\mu_m}^m
  \qquad\mbox{for}\qquad
  \tau_{\alpha_m\beta_m\mu_m}^m:=\int_\Gamma P_{\alpha_m}^{m}P_{\beta_m}^{m}P_{\mu_m}^{m}\dd\pi_{\vartheta\rho,m}(y)
\end{equation}
and note that $\tau_{ijk}^m=0$ if $i+j+k$ is odd or if $\max\{ i,j,k\} > (i+j+k)/2$.
For more analytical properties of the scaled Hermite polynomials we refer to~\cite{EigelMarschall2020lognormal,EigelFarchmin2022avmc}.

For any $j\in\mathbb{N}_0$ and $k\in\mathbb{N}$, let $[j\sep k]:=\{ j,\dots,k-1 \}$, where $[j\sep k] := \{ 0 \}$ if $j\geq k$ and $[k] := [0\sep k-1]$.
Define the full tensor index set $\Lambda_d$ for any $d\in\mathbb{N}^{\hat M}$ by
\begin{equation}
  \label{eq:setting:Lambda}
  \Lambda_d := [d_1] \times \dots [d_{\hat M}] \times [1] \times \dots \subset \mathcal{F}.
\end{equation}
Given two full tensor sets $\Lambda_d,\Lambda_{\hat d}$ with $d,\hat d\in\mathbb{N}^{\hat M}$, we define the sum of the two sets by $\Lambda_d+\Lambda_{\hat d} = \Lambda_{d+\hat d-1}$ and refer to $\partial \Lambda_d =\partial \Lambda_{d,\hat d} = \Lambda_{d+\hat d-1}\setminus\Lambda_d$ as the boundary of $\Lambda_d$ with respect to $\Lambda_{\hat d}$, where we drop the second subscript if the dependence on $\hat{d}$ is clear.
Note that we implicitly assume $d = (d_1,\dots,d_M, 1, \dots, 1)\in\mathbb{N}^{\hat M}$ for any $d\in\mathbb{N}^M$ with $M< \hat M$ to ensure compatibility of dimensions.

Given these sets, we define the fully discrete approximation space by
\begin{equation}
  \label{eq:setting:V_N}
  \mathcal{V}_N := \mathcal{V}_N(\Lambda_d; \mathcal{T}, p)
  := \Bigl\{ w_N(x,y) = \sum_{j=1}^{J}\sum_{\mu\in\Lambda_d} \discretized{w_{N}}[j,\mu] \varphi_j(x) P_\mu(y) \mbox{ with } \discretized{w_{N}}\in\mathbb{R}^{J\times d}\Bigr\}
  \subset \mathcal{V}.
\end{equation}
We refer to $\discretized{w_N}$ as the coefficient tensor of $w_N$ with respect to the bases $\{ \varphi_j \}$ and $\{ P_\mu \}$.
With this the Galerkin projection $u_N\in\mathcal{V}_N$ of the solution $u$ of~\eqref{eq:setting:varriational_form} is determined uniquely by
\begin{equation}
  \label{eq:setting:galerkin_system}
  B(u_N,v_N) = F(v_N)
\qquad\mbox{for all }v_N\in\mathcal{V}_N.
\end{equation}

\section{Error estimator and adaptive algorithm}%
\label{sec:estimator}

In this section a reliable residual based error estimator for the Galerkin solution $u_N\in \mathcal{V}_N$ is introduced.
Moreover, an algorithm for the adaptive refinement of the spatial and stochastic approximation spaces is described.
The estimator we define is based on the one presented in~\cite{EigelMarschall2020lognormal} and we refer to~\cite{EigelGittelson2014asgfem,EigelGittelson2014convergence,EigelMarschall2020lognormal} for details on the underlying analysis.

\subsection{A posteriori error estimator}%
\label{sec:estimator:estimator}

In the following, let $p\in\mathbb{N}$ be some fixed FE polynomial degree, $d\in\mathbb{N}^{M}$ and $\hat d\in\mathbb{N}^{\hat M}$ for $M \leq \hat M \in \mathbb{N}$.
Furthermore, assume that we have access to an approximation $a_N\in \mathcal{V}_N(\Lambda_{\hat d};\mathcal{T},p)$ of~\eqref{eq:setting:a} of the form
\begin{equation}
  \label{eq:estimator:a_N}
  a(x,y) 
  \approx a_N(x,y)
  = \sum_{j=1}^J\sum_{\hat\mu\in\Lambda_{\hat d}} \discretized{a_N}[j,\hat\mu]\varphi_j(x)P_{\hat\mu}(y)
\end{equation}
for any triangulation $\mathcal{T}$.
Then, for any $w\in \mathcal{V}$, consider the residual $\mathcal{R}(w)\in \mathcal{V}^* = L^2(\Gamma,\pi_{\vartheta\rho};\mathcal{X}^*)$ of~\eqref{eq:setting:varriational_form} for $a_N$ instead of $a$, given implicitly for all $v\in \mathcal{V}$ by
\begin{equation}
  \label{eq:estimator:residual}
  \langle \mathcal{R}(w), v \rangle_{V^*, V}
  = \int_{\Gamma} \int_{D} f(x)v(x,y) - a_N(x,y) \grad w(x,y)\cdot\grad v(x,y) \,\mathrm{d}x \,\mathrm{d}\pi_{\vartheta\rho}(y).
\end{equation}
Similarily, we consider the bilinear form~\eqref{eq:setting:B} with $a_N$ instead of $a$ in the following.

\begin{remark}%
  \label{rem:estimator:use_a_N}
  The residual associated to~\eqref{eq:setting:varriational_form} can be split into a term containing $\mathcal{R}(w)$ and a term describing the approximate error of the diffusion coefficient
  \begin{equation}
    \label{eq:estimator:appriximative_residual}
    \langle f + \ddiv(a\grad w), v\rangle_{\mathcal{V}^*, V}
    = \langle \mathcal{R}(w), v\rangle_{\mathcal{V}^*, V} + \int_{\Gamma}\int_{D} (a-a_N)\grad w\cdot\grad v \,\mathrm{d}x  \,\mathrm{d}\pi_{\vartheta\rho}(y),
  \end{equation}
  which implies that the task of finding $a_N$ to approximate~\eqref{eq:setting:a} can be considered separately.
  Since there exist several approaches to compute such an approximation~\cite{Espig2014, DolgovScheichl2019alsCross,EigelMarschall2020lognormal,EigelFarchmin2021expTT} and~\cite{EigelFarchmin2021expTT} even provides a computable a posteriori bound of the approximation error for any $a_N\in \mathcal{V}_N(\Lambda_{\hat d};\mathcal{T},p)$, we assume that the second term of~\eqref{eq:estimator:appriximative_residual} is negligible (with respect to the accurracy we would like to achieve for solution $u_N$) and restrict ourselves to~\eqref{eq:estimator:residual}.
  We also need to assume that~\eqref{eq:setting:boundedness_B}--~\eqref{eq:setting:coercivity_B} still hold when using $a_N$ instead of $a$, possibly with different constants $\hat c_{\vartheta\rho}$ and $\check c_{\vartheta\rho}$.
  Note that the solution error incurred by approximating $u$ by $u_N$ is bounded by the Strang lemma, see e.g.~\cite[Chap.~3-§1]{Braess2007}.
\end{remark}

For any $w_N\in \mathcal{V}_N$ we define $r(w_N) := a_N\grad w_N$ and note that by~\cite[Lemma 3.1]{EigelMarschall2020lognormal}
\begin{equation}
  \label{eq:estimator:r}
  r(w_N) = \sum_{\mu\in\Lambda_{d+\hat d-1}} r_\mu(w_N)P_\mu
  \quad\mbox{for}\quad
  r_\mu(w_N) = \sum_{j,k=1}^{J} \sum_{\alpha\in\Lambda_d}\sum_{\hat\alpha\in\Lambda_{\hat d}} \discretized{a_N}[k,\hat\alpha]\discretized{w_N}[j,\alpha]\tau_{\alpha\hat\alpha\mu}\varphi_k\grad\varphi_j
\end{equation}
has support on $\Lambda_{d+\hat d-1}$ due to the properties of the triple product $\tau_{\alpha\hat\alpha\mu}$.
With this we introduce the deterministic estimator contribution
\begin{equation}
  \label{eq:estimator:eta_det}
  \eta_\mathrm{det}(w_N,\mathcal{T},\Lambda_d)^2
  := \sum_{T\in\mathcal{T}} \Bigl(\eta_{\mathrm{det},T}(w_N,\Lambda_d)^2 + \eta_{\mathrm{det},\partial T}(w_N,\Lambda_d)^2\Bigr)
\end{equation}
for volume and jump contributions given by
\begin{align}
  \label{eq:estimator:eta_det:vol}
  \eta_{\mathrm{det},T}(w_N,\Lambda_d) 
  &:= h_T \Vert \sum_{\mu\in\Lambda_d}\bigl(f\delta_{\mu0} - \ddiv r_\mu(w_N) \bigr) P_\mu\,\zeta_{\vartheta\rho}\Vert_{\pi_0,T},\\
  \label{eq:estimator:eta_det:jump}
  \eta_{\mathrm{det},\partial T}(w_N,\Lambda_d) 
  &:= h_T^{1/2} \Vert \sum_{\mu\in\Lambda_d}\llbracket r_\mu(w_N) \rrbracket_{\partial T} P_\mu\,\zeta_{\vartheta\rho}\Vert_{\pi_0,\partial T},
\end{align}
where $\delta_{\mu\hat\mu}$ denotes the (multiindex) Kronecker delta.
In a similar fashion we define the stochastic estimator contribution by
\begin{equation}
  \label{eq:estimator:eta_sto}
  \eta_\mathrm{sto}(w_N,\Delta)
  := \Vert \sum_{\mu\in\Delta} r_\mu(w_N) P_\mu \zeta_{\vartheta\rho}\Vert_{\pi_0,D}
  \qquad\mbox{for any }\qquad\Delta\subseteq\partial\Lambda_d.
\end{equation}
Let the overall error estimator for any $w_N\in\mathcal{V}_N$ and (heuristically chosen\footnote{this is required since the scaling of the standard residual based error estimator cannot be computed a priori}) equilibration constant $c_{\mathrm{eq}} > 0$ be given by
\begin{equation}
  \label{eq:estimator:eta}
  \eta(w_N)^2
  = \eta_\mathrm{det}(w_N,\mathcal{T},\Lambda_d)^2 + c_\mathrm{eq}^2\eta_\mathrm{sto}(w_N,\partial\Lambda_d)^2.
\end{equation}
Then~\cite{EigelMarschall2020lognormal} yields the following reliability bound.

\begin{theorem}[reliability]%
  \label{thm:estimator:reliability}
  Let $u$ be the solution of~\eqref{eq:setting:varriational_form} and let $u_N\in \mathcal{V}_N$ be the Galerkin solution of~\eqref{eq:setting:galerkin_system}.
  There exists a constant $c_{\mathrm{rel}}>0$ depending only on the constant in~\eqref{eq:setting:coercivity_B} and the shape regularity of $\mathcal{T}$ such that
  \begin{equation*}
    \Vert u-u_N\Vert_B \leq c_{\mathrm{rel}}\,\eta(u_N).
  \end{equation*}
\end{theorem}

\begin{remark}%
  The estimator considered in~\cite{EigelMarschall2020lognormal} is not the same as~\eqref{eq:estimator:eta} for two reasons.
  First,~\eqref{eq:estimator:eta_det:jump} has a different ordering, which implies that~\eqref{eq:estimator:eta_det} is equivalent to the deterministic estimator contribution in~\cite{EigelMarschall2020lognormal} with a constant depending only on the shape regularity of $\mathcal{T}$.
  Second, we neglect error contributions introduced by the low-rank compression and the inexact iterative solver of the Galerkin solution $u_N$ of~\eqref{eq:setting:galerkin_system}.
  This is motivated by promising experimental results in recent works~\cite{EigelTrunschke2019vmc,EigelTrunschke2020,Trunschke2021}, where it is shown numerically that these errors can in principle be controlled, the inclusion of which would however complicate the analysis unnecessarily.
  Hence, we assume that the Galerkin solution $u_N$ is computable and focus on showing the reduction of the quasi-error through the adaptive algorithm in our setting.
\end{remark}

\subsection{Adaptive algorithm}%
\label{sec:estimator:algorithm}

Given a fixed FE polynomial degree $p\in\mathbb{N}_0$, an initial triangulation $\mathcal{T}$ and initial stochastic dimensions $d\in\mathbb{N}^M$, Algorithm~\ref{alg:estimator:adaptive} relies on a classical loop of \emph{Solve}--\emph{Estimate}--\emph{Mark}--\emph{Refine} to generate approximative solutions $u_\ell$ of~\eqref{eq:introduction:darcy}.
According to Remark~\ref{rem:estimator:use_a_N}, we assume $\hat d\in\mathbb{N}^{\hat M}$ sufficiently large such that the approximation error of the discretized diffusion coefficient~\eqref{eq:estimator:a_N} is negligible in each iteration of Algorithm~\ref{alg:estimator:adaptive}.

In each iteration $\ell$ of the loop, we compute the Galerkin solution $u_\ell$ of~\eqref{eq:setting:galerkin_system} (\emph{Solve}) and the estimator contributions~\eqref{eq:estimator:eta_det} and~\eqref{eq:estimator:eta_sto} (\emph{Estimate}).
During the \emph{Mark} step, we employ a conditional D\"orfler marking strategy based on the dominating estimator contribution to determine how to refine $\mathcal{T}_\ell$ and $\Lambda_{d_\ell}$.
If, on the one hand, $\eta_{\mathrm{det},\ell}(u_\ell,\mathcal{T}_\ell,\Lambda_{d_\ell})$ dominates we set the marked stochastic set $\Delta_{\ell}=\emptyset$ and employ a D\"orfler marking strategy to the mesh $\mathcal{T}_\ell$, where we use the D\"orfler threshold $\theta_{\mathrm{det}}$.
For the refinement of the spatial mesh $\mathcal{T}_\ell$, newest vertex bisection~\cite{Stevenson2008} is used on all marked elements $T\in\mathcal{M}_\ell$, which is denoted by ${\bf bisect}(\mathcal{T}_\ell,\mathcal{M}_\ell)$ in Algorithm~\ref{alg:estimator:adaptive}.
If, on the other hand, $c_{\mathrm{eq}}\,\eta_{\mathrm{sto},\ell}(u_\ell,\partial\Lambda_{d_\ell})$ dominates the error estimator, where $c_{\mathrm{eq}}$ is the equilibration constant in~\eqref{eq:estimator:eta}, we set the marked triangles to $\mathcal{M}_\ell=\emptyset$ and employ the D\"orfler marking with threshold $\theta_{\mathrm{sto}}$ to \{$\Delta_{\ell,m,q_m}\}_{m=1}^{\hat M}$.
The index sets $\Delta_{\ell,m,q_m}$ are given for each $m=1,\dots,\hat M$ by
\begin{equation}
  \label{eq:estimator:Delta_m}
  \Delta_{\ell,m,q_m} 
  = \bigotimes_{j=1}^{m-1}[d_j] \ \otimes\ [d_m:d_m+q_m] \ \otimes\ \bigotimes_{j=m+1}^{\hat M}[d_j]
  \qquad\mbox{for }q_m\in[1\sep \hat d_m-1].
\end{equation} 
Here the look ahead $q\in\mathbb{N}^{\hat M}$ is a free parameter which determines the depth of the considered univariate boundary contributions $\Delta_{\ell,m,q_m}$~\cite{EigelFarchmin2022avmc}.
Algorithm~\ref{alg:estimator:adaptive} iterates through the \emph{Solve}, \emph{Estimate}, \emph{Mark} and \emph{Refine} routines until the maximal number of loops is reached.
Note that specifying the maximum number of iterations a priori in Algorithm~\ref{alg:estimator:adaptive} is impractical for most applications and done here only for simplification.
More reasonable stopping criteria such as specifying a target threshold for the total error estimator or limiting the maximum of allowed degrees of freedom could be applied as well.
Different refinement strategies were explored in~\cite{BespalovPraetorius2019convergence}.

\medskip
\noindent
\begin{minipage}{\linewidth}
\begin{algorithm2e}[H]\label{alg:estimator:adaptive}
  \caption{Adaptive refinement scheme}
    \KwInput{%
        mesh $\mathcal{T}$;
        FE polynomial degree $p$;
        stochastic dimensions $d$;
        D\"orfler thresholds $0 < \theta_{\mathrm{det}}, \theta_{\mathrm{sto}} \leq 1$;
        maximum number of iterations $L$;
        look ahead $q\in\mathbb{N}^{\hat M}$
    }
    set $\mathcal{T}_1=\mathcal{T}$, $\Lambda_{d_1} = \Lambda_d$;\\
    \For{$\ell = 1, 2, \dots$}{%
      \Solve{%
        compute solution $u_\ell\in\mathcal{V}_N(\Lambda_{d_\ell}; \mathcal{T}_\ell, p)$ of~\eqref{eq:setting:galerkin_system};
      }

      \Estimate{%
        compute $\eta_{\mathrm{det},\ell}(u_\ell,\mathcal{T}_\ell,\Lambda_{d_\ell})$ as in~\eqref{eq:estimator:eta_det};\\
        compute $\eta_{\mathrm{sto},\ell}(u_\ell,\partial\Lambda_{d_\ell})$ as in~\eqref{eq:estimator:eta_sto};\\
      }

      {\bf if} $\ell\geq L$: {\bf break}\\

      \Mark{%
        \eIf{$\eta_{\mathrm{det},\ell}(u_\ell,\mathcal{T}_\ell,\Lambda_{d_\ell}) \geq c_{\mathrm{eq}}\,\eta_{\mathrm{sto},\ell}(u_\ell,\partial\Lambda_{d_\ell})$}{%
          set $\Delta_\ell=\emptyset$ and choose minimal set $\mathcal{M}_\ell\subseteq \mathcal{T}_\ell$ such that $\eta_{\mathrm{det},\ell}(u_\ell,\mathcal{M}_\ell,\Lambda_{d_\ell}) \geq \theta_{\mathrm{det}}\,\eta_{\mathrm{det},\ell}(u_\ell,\mathcal{T}_\ell,\Lambda_{d_\ell})$;
        }{%
          choose minimal set $\mathcal{M}\subseteq\{ 1,\dots,\hat M \}$ such that $\sum_{m\in \mathcal{M}}\eta_{\mathrm{sto},\ell}(u_\ell, \Delta_{\ell,m,q_m}) \geq \theta_{\mathrm{sto}} \, \eta_{\mathrm{sto},\ell}(u_\ell,\partial\Lambda_{d_\ell})$;\\
          set $\mathcal{M}_\ell=\emptyset$ and $\Delta_\ell = \bigcup_{m\in \mathcal{M}} \Delta_{\ell,m,q_m}$;
        }
      }

      \Refine{%
        $\mathcal{T}_{\ell+1} \gets \mathrm{\bf bisect}(\mathcal{T}_\ell, \mathcal{M}_\ell)$;\\
        $\Lambda_{d_{l+1}} \gets \Lambda_{d_\ell} \cup \Delta_\ell$;
      }
    }
\end{algorithm2e}
\end{minipage}
\medskip

\section{Estimator properties}%
\label{sec:properties}

In this section we establish some fundamental properties of the estimator contributions, which are required to prove the quasi-error reduction of Algorithm~\ref{alg:estimator:adaptive} for the unbounded lognormal diffusion coefficient~\eqref{eq:setting:a}.
The proof follows~\cite{EigelGittelson2014convergence} and is based on the continuity of the estimators.
As a preparation, we first establish how $\zeta_\rho^\alpha\pi_0$ behaves for different exponents $\alpha\geq 0$.
\begin{lemma}[embedding of weighted $L^p$-spaces]%
  \label{lem:properties:weighted_gaussian_measure}
  Let $\alpha \geq 0$ and 
  \begin{equation*}
    0 \leq \rho \leq \rho_\alpha =
    \begin{cases}
      1 & \mbox{if }\alpha \leq 1,\\
      \min\Bigl\{ 1, \frac{\ln(\alpha(\alpha-1)^{-1})}{2\hat\gamma} \Bigr\} & \mbox{if }\alpha>1,
    \end{cases}
  \end{equation*}
  where $\hat\gamma$ is an upper bound for~\eqref{eq:setting:gamma}, i.e., $\Vert \gamma_m \Vert_{L^\infty(D)}\leq\hat\gamma$ for all $m=0,\dots,\hat M$.
  Then there exists a constant $c_\alpha>0$ such that $\pi_{\rho,\alpha} = c_\alpha\zeta_{\rho}^\alpha\pi_0$ is the probability density function of a centered Gaussian distribution.
  Moreover, for any $0\leq\alpha\leq\beta$ and $\rho\leq\rho_\beta$, it holds
  \begin{equation*}
    L^p(\Gamma,\pi_{\rho,\beta}) \subset L^p(\Gamma,\pi_{\rho,\alpha})
    \qquad\mbox{for any }p\in[1,\infty).
  \end{equation*}
\end{lemma}

\begin{proof}%
  Note that for any $\alpha>0$ and $\sigma_{\alpha,m}(\rho) = \sigma_m(\rho)\bigl( \alpha+(1-\alpha)\sigma_m(\rho)^2 \bigr)^{-1/2}$ it holds
  \begin{equation*}
    \begin{split}
      \zeta_{\rho,m}(y_m)^\alpha \pi_{0,m}(y_m)
      &= \frac{1}{\sqrt{2\pi}\sigma_m(\rho)^\alpha}\exp\Bigl( -\frac{y_m^2}{2\sigma_{\alpha,m}(\rho)^2} \Bigr).
    \end{split}
  \end{equation*}
  Defining the normalization constant
  \begin{equation*}
    c_{\alpha,m}(\rho)
    := \sigma_m(\rho)^{\alpha-1}\sqrt{\alpha+(1-\alpha)\sigma_m(\rho)^2}
    = \sigma_{m}(\rho)^\alpha\sigma_{\alpha,m}(\rho)^{-1}
  \end{equation*}
  then yields that $\pi_{\rho,\alpha,m} = c_{\alpha,m}(\rho) \zeta_{\rho,m}^\alpha \pi_{0,m}$ is the probability density function of a univariate centered Gaussian distribution with variance $\sigma_{\alpha,m}(\rho)^2$.
  Since $\sigma_m(\rho)=\exp(\rho\Vert\gamma_m\Vert_{L^\infty(D)})$, we need
  \begin{equation*}
    0 < \rho <
    \begin{cases}
      1 & \mbox{if }\alpha \leq 1,\\
      \min\Bigl\{ 1, \frac{\ln(\alpha(\alpha-1)^{-1})}{2\Vert \gamma_m \Vert_{L^\infty(D)}} \Bigr\} & \mbox{if }\alpha>1,
    \end{cases}
  \end{equation*}
  to ensure that $\pi_{\rho,\alpha,m}$ is integrable with respect to the standard Lebesgue measure.
  As a consequence and by the definition of $\hat\gamma$, $\rho\leq\rho_{\alpha}$ ensures integrability of $\pi_{\rho,\alpha} = \prod_{m=1}^{\hat M} \pi_{\rho,\alpha,m}$.
  It is easy to see that for any $0<\alpha<\beta\in\mathbb{R}$ and $w\in L^p(\Gamma,\pi_{\rho,\beta})$,
  \begin{equation*}
    \Vert w \Vert_{L^p(\Gamma,\pi_{\rho,\alpha})}
    \leq c\,\Vert \zeta_{\rho}^{\alpha-\beta} \Vert_{L^\infty(\Gamma)} \Vert w \Vert_{L^p(\Gamma,\pi_{\rho,\beta})}
    \quad\text{with}\quad c = \prod_{m=1}^{\hat M} \frac{c_{\alpha,m}(\rho)}{c_{\beta,m}(\rho)}.
  \end{equation*}
  Moreover, $\vert\zeta_{\rho,m}^{\alpha-\beta}\vert\leq\sigma_m(\rho)^{\beta-\alpha}<\infty$, which implies
  \begin{equation*}
    \Vert \zeta_{\rho}^{\alpha-\beta} \Vert_{L^\infty(\Gamma)}
    = \prod_{m=1}^{\hat M}\sigma_m(\rho)^{\beta-\alpha}
    <\infty.
  \end{equation*}
  Consequently, $L^p(\Gamma,\pi_{\rho,\beta}) \subset L^p(\Gamma,\pi_{\rho,\alpha})$ for all $1\leq p\in\mathbb{R}$.
\end{proof}

To prove the Lipschitz continuity of $\eta_{\mathrm{det}}$ and $\eta_{\mathrm{sto}}$ later in this section, we also require the following observations.

\begin{corollary}%
  \label{cor:properties:polynom_integration}
  Let $w$ be a polynomial, $\alpha\geq0$ and $\rho\leq\rho_{\alpha}$ as in Lemma~\ref{lem:properties:weighted_gaussian_measure}. Then it holds
  \begin{equation*}
    \Vert w \zeta_{\rho}^\alpha\Vert_{L^p(\Gamma, \pi_{0})} < \infty.
  \end{equation*}
\end{corollary}

\begin{proof}%
  By Lemma~\ref{lem:properties:weighted_gaussian_measure} there exists a constant $c>0$ such that $\Vert w \zeta_{\rho}^\alpha\Vert_{L^p(\Gamma, \pi_{0})}^p = c \Vert w \Vert_{L^p(\Gamma, \pi_{\rho,p\alpha})}^p$.
  Since $w$ is a polynomial and $\pi_{\rho,p\alpha}$ decays exponentially, $\Vert w \Vert_{L^p(\Gamma, \pi_{\rho,p\alpha})}<\infty$.
\end{proof}

\begin{remark}[doubly orthogonal polynomials]%
  \label{rem:properties:doubly_orthogonal_polynomials}
  For any index set $\Lambda\subset\mathbb{N}_0^M$ with basis $\{ P_\mu \}_{\mu\in\Lambda}$, there exists a linear transformation $Z\colon\operatorname{span}\{ P_\mu \ \vert\ \mu\in\Lambda \} \to \operatorname{span}\{ P_\mu \ \vert\ \mu\in\Lambda \}$ such that the inverse transformation $Z^{-1}$ maps $\{ P_\mu \}$ onto a doubly orthogonal polynomial basis $\{ \hat{P}_\mu \}$, i.e.\ such that it holds
  \begin{equation*}
    \int_\Gamma \hat{P}_\mu\hat{P}_\nu\zeta_{\vartheta\rho}\dd\pi_0 = \delta_{\mu\nu}
    \qquad\mbox{and}\qquad
    \int_\Gamma \hat{P}_\mu\hat{P}_\nu\zeta_{\vartheta\rho}^2\dd\pi_0 = c_\mu^2\delta_{\mu\nu},
  \end{equation*}
  for all $\mu,\nu\in\Lambda_d$, where $0 < c_\mu \to \infty$ as $\vert\mu\vert\to\infty$, see~\cite{Shapiro1979} .
  If $\vert\Lambda\vert < \infty$ the transformation can be written as $Z\in\mathbb{R}^{\vert\Lambda\vert\times\vert\Lambda\vert}$ such that $P_\mu = \sum_{\nu\in\Lambda} z_{\mu\nu} \hat{P}_\nu$.
\end{remark}

\subsection{Deterministic estimator contribution}%
\label{sec:properties:det}

In the following, we show that the deterministic estimator contribution~\eqref{eq:estimator:eta_det} satisfies a continuity conditions.

\begin{theorem}[Lipschitz continuity of $\eta_{\mathrm{det}}$ in the first component on each $T$]%
  \label{thm:properties:lipschitz_det}
  For any $v_N,w_N\in\mathcal{V}_N=\mathcal{V}_N(\Lambda_d;\mathcal{T},p)$ and $T\in\mathcal{T}$ there exists a constant $c_{\mathrm{det}}>0$ depending only on the active set $\Lambda_d$, such that
  \begin{equation*}
    \bigl\vert \left.\eta_{\mathrm{det}}(v_N,\Lambda_d)\right|_T - \left.\eta_{\mathrm{det}}(w_N,\Lambda_d)\right|_T\bigr\vert
    \leq c_{\mathrm{det}} \Vert \grad(v_N-w_N)\Vert_{\pi_{\vartheta\rho},T}.
  \end{equation*}
\end{theorem}

\begin{proof}%
  Before we show the claim, we first need to some intermediate results.
  First we note that with the definitions~\eqref{eq:estimator:eta_det:vol} and~\eqref{eq:estimator:eta_det:jump} and Remark~\ref{rem:properties:doubly_orthogonal_polynomials}, it holds
  \begin{equation}%
    \label{eq:properties:lipschitz_det:proof:a}
    \eta_{\mathrm{det},T}(v_N, \Lambda_d)^2
    = h_T^2 \sum_{\nu\in\Lambda_d} c_\nu^2 \Vert \sum_{\mu\in\Lambda_d} (f\delta_{\mu 0} - \ddiv r_\mu(v_N)) z_{\mu\nu} \Vert_T^2
  \end{equation}
  and
  \begin{equation}%
    \label{eq:properties:lipschitz_det:proof:b}
    \eta_{\mathrm{det},\partial T}(v_N, \Lambda_d)^2
    = h_T \sum_{\nu\in\Lambda_d} c_\nu^2 \Vert \sum_{\mu\in\Lambda_d} \llbracket r_\mu(v_N) \rrbracket_{\partial T} z_{\mu\nu} \Vert_{\partial T}^2.
  \end{equation}
  Secondly, the definition of $r_\mu$~\eqref{eq:estimator:r} and the inverse estimates~\eqref{eq:setting:inverse_estimates} in combination with the triangle inequality yield
  \begin{equation}%
    \label{eq:properties:lipschitz_det:proof:c}
    \begin{aligned}
        &h_T \Vert \sum_{\mu\in\Lambda_d} \ddiv r_\mu(w_N-v_N) z_{\mu\nu} \Vert_T\\
        &\qquad\leq \sum_{\beta\in\Lambda_d} \Vert \grad(w_N-v_N)_\beta\Vert_T \Bigl( c_{\mathrm{inv}}\sum_{\alpha\in\Lambda_{d+\hat{d}-1}}\sum_{\mu\in\Lambda_d} \Vert a_\alpha \Vert_{L^\infty(D)} \vert \tau_{\alpha\beta\mu}z_{\mu\nu}\vert \Bigr)
    \end{aligned}
  \end{equation}
  and
  \begin{equation}%
    \label{eq:properties:lipschitz_det:proof:d}
    \begin{aligned}
        & h_T^{1/2} \Vert \sum_{\mu\in\Lambda_d} \llbracket r_\mu(w_N-v_N) \rrbracket_{\partial T} z_{\mu\nu} \Vert_{\partial T}\\
        &\qquad\leq \sum_{\beta\in\Lambda_d} \Vert \grad(w_N-v_N)_\beta\Vert_T \Bigl( c_{\mathrm{inv}}\sum_{\alpha\in\Lambda_{d+\hat{d}-1}}\sum_{\mu\in\Lambda_d} \Vert a_\alpha \Vert_{L^\infty(D)} \vert \tau_{\alpha\beta\mu}z_{\mu\nu}\vert \Bigr),
    \end{aligned}
  \end{equation}
  where we denote $a_\alpha = \sum_{k=1}^{J} \discretized{a}_N[k,\alpha]\varphi_k\in\mathcal{X}_p(\mathcal{T})$.
  Next we define for any $\nu\in\Lambda_d$
  \begin{equation*}
    s_{T,\nu} 
    := \Vert \sum_{\mu\in\Lambda_d} (f\delta_{\mu 0} - \ddiv r_\mu(v_N)) z_{\mu\nu}\Vert_T
    + \Vert \sum_{\mu\in\Lambda_d} (f\delta_{\mu 0} - \ddiv r_\mu(w_N)) z_{\mu\nu}\Vert_T
  \end{equation*}
  and
  \begin{equation*}
    s_{\partial T,\nu} 
    := \Vert \sum_{\mu\in\Lambda_d} \llbracket r_\mu(v_N) \rrbracket_{\partial T} z_{\mu\nu}\Vert_{\partial T}
    + \Vert \sum_{\mu\in\Lambda_d} \llbracket r_\mu(w_N) \rrbracket_{\partial T} z_{\mu\nu}\Vert_{\partial T}.
  \end{equation*}
  Using~\eqref{eq:properties:lipschitz_det:proof:a} and~\eqref{eq:properties:lipschitz_det:proof:b}, respectively, in combination with the triangle inequality, the third binomial formula, the inverse triangle inequality and~\eqref{eq:properties:lipschitz_det:proof:c} and~\eqref{eq:properties:lipschitz_det:proof:d} give
  \begin{align*}%
    & \vert \eta_{\mathrm{det},T}(v_N,\Lambda_d)^2 - \eta_{\mathrm{det},T}(w_N,\Lambda_d)^2\vert \notag\\
    &\qquad\leq h_T^2 \sum_{\nu\in\Lambda_d} c_\nu^2 \, s_{T,\nu}\, \Vert \sum_{\mu\in\Lambda_d} \ddiv r_\mu(w_N-v_N) z_{\mu\nu}\Vert_T \notag\\
    &\qquad\leq \sum_{\beta\in\Lambda_d} \Vert \grad(w_N-v_N)_\beta\Vert_T \Bigl( \sum_{\alpha\in\Lambda_{d+\hat{d}-1}}\sum_{\mu\in\Lambda_d} c_{\mathrm{inv}} \Vert a_\alpha \Vert_{L^\infty(D)} \vert \tau_{\alpha\beta\mu}\vert \sum_{\nu\in\Lambda_d} c_\nu^2 \,  h_T s_{T,\nu} \, \vert z_{\mu\nu}\vert \Bigr)
  \end{align*}
  and
  \begin{align*}%
    & \vert \eta_{\mathrm{det},\partial T}(v_N,\Lambda_d)^2 - \eta_{\mathrm{det},\partial T}(w_N,\Lambda_d)^2\vert \notag\\
    &\qquad\leq h_T \sum_{\nu\in\Lambda_d} c_\nu^2 \, s_{\partial T,\nu}\, \Vert \sum_{\mu\in\Lambda_d} \llbracket r_\mu(w_N-v_N) \rrbracket_{\partial T} z_{\mu\nu}\Vert_T \notag\\
    &\qquad\leq \sum_{\beta\in\Lambda_d} \Vert \grad(w_N-v_N)_\beta\Vert_T \Bigl( \sum_{\alpha\in\Lambda_{d+\hat{d}-1}}\sum_{\mu\in\Lambda_d} c_{\mathrm{inv}} \Vert a_\alpha \Vert_{L^\infty(D)} \vert \tau_{\alpha\beta\mu}\vert \sum_{\nu\in\Lambda_d} c_\nu^2 \,  h_T^{1/2} s_{\partial T,\nu} \, \vert z_{\mu\nu}\vert \Bigr),
  \end{align*}
  which combines to
  \begin{align}%
    \label{eq:properties:lipschitz_det:proof:e}
    & \vert \eta_{\mathrm{det},T}(v_N,\Lambda_d)^2 - \eta_{\mathrm{det},T}(w_N,\Lambda_d)^2\vert  + \vert \eta_{\mathrm{det},\partial T}(v_N,\Lambda_d)^2 - \eta_{\mathrm{det},\partial T}(w_N,\Lambda_d)^2\vert \notag\\
    &\qquad\leq \sum_{\beta\in\Lambda_d} \Vert \grad(w_N-v_N)_\beta\Vert_T \Bigl( \sum_{\alpha\in\Lambda_{d+\hat{d}-1}}\sum_{\mu\in\Lambda_d} c_{\mathrm{inv}} \Vert a_\alpha \Vert_{L^\infty(D)} \vert \tau_{\alpha\beta\mu}\vert \sum_{\nu\in\Lambda_d} c_\nu^2 \,  (h_T s_{T,\nu} + h_T^{1/2} s_{\partial T,\nu}) \, \vert z_{\mu\nu}\vert \Bigr).
  \end{align}
  Note that by the Cauchy-Schwarz inequality, triangle inequality, Remark~\ref{rem:properties:doubly_orthogonal_polynomials} and~\eqref{eq:properties:lipschitz_det:proof:a}--\eqref{eq:properties:lipschitz_det:proof:b} we have the estimate
  \begin{align*}
    \sum_{\nu\in\Lambda_d} c_\nu^2 \,  (h_T s_{T,\nu} + h_T^{1/2} s_{\partial T,\nu}) \, \vert z_{\mu\nu}\vert
    &\leq \Bigl( \sum_{\nu\in\Lambda_d} c_\nu^2 \vert z_{\mu\nu}\vert^2 \Bigr)^{1/2} \Bigl( \sum_{\nu\in\Lambda_d} c_\nu^2 \,  (h_T s_{T,\nu} + h_T^{1/2} s_{\partial T,\nu})^2 \Bigr)^{1/2}\\
    &\leq 2 \Vert P_\mu\zeta_{\vartheta\rho}\Vert_{\pi_0} \Bigl( \left.\eta_{\mathrm{det}}(v_N,\Lambda_d)\right|_T + \left.\eta_{\mathrm{det}}(w_N,\Lambda_d)\right|_T \Bigr).
  \end{align*}
  The combination of this estimate, the Cauchy-Schwarz inequality,~\eqref{eq:properties:lipschitz_det:proof:e} and the definition
  \begin{equation}
    \label{eq:properties:lipschitz_det:proof:constant}
    c(\Lambda_d)^2 
    := \sum_{\beta\in\Lambda_d} \Bigl( \sum_{\mu\in\Lambda_d}\sum_{\alpha\in\Lambda_{d+\hat{d}-1}}  \Vert a_\alpha\Vert_{L^\infty(D)}\vert\tau_{\alpha\beta\mu}\vert \Vert P_\mu\zeta_{\vartheta\rho}\Vert_{\pi_0} \Bigr)^2
  \end{equation}
  now yield
  \begin{align}%
    \label{eq:properties:lipschitz_det:proof:f}
    &\vert \eta_{\mathrm{det},T}(v_N,\Lambda_d)^2 - \eta_{\mathrm{det},T}(w_N,\Lambda_d)^2\vert  + \vert \eta_{\mathrm{det},\partial T}(v_N,\Lambda_d)^2 - \eta_{\mathrm{det},\partial T}(w_N,\Lambda_d)^2\vert \notag\\
    &\qquad\leq 2 c_{\mathrm{inv}}c(\Lambda_d)\Bigl( \left.\eta_{\mathrm{det}}(v_N,\Lambda_d)\right|_T + \left.\eta_{\mathrm{det}}(w_N,\Lambda_d)\right|_T \Bigr)\Bigl(\sum_{\beta\in\Lambda_d} \Vert \grad(w_N-v_N)_\beta\Vert_T^2\Bigr)^{1/2}\notag\\
    &\qquad= 2 c_{\mathrm{inv}}c(\Lambda_d)\Bigl( \left.\eta_{\mathrm{det}}(v_N,\Lambda_d)\right|_T + \left.\eta_{\mathrm{det}}(w_N,\Lambda_d)\right|_T \Bigr)\Vert \grad(w_N-v_N)\Vert_{\pi_{\vartheta\rho}, T}.
  \end{align}
  With~\eqref{eq:properties:lipschitz_det:proof:f} established, we conclude the proof with the simple computation
  \begin{align*}%
    &\Bigl\vert \left.\eta_{\mathrm{det}}(v_N,\Lambda_d)\right|_T - \left.\eta_{\mathrm{det}}(w_N,\Lambda_d)\right|_T \Bigr\vert \Bigl( \left.\eta_{\mathrm{det}}(v_N,\Lambda_d)\right|_T + \left.\eta_{\mathrm{det}}(w_N,\Lambda_d)\right|_T \Bigr) \\
    &\qquad = \left\vert \left.\eta_{\mathrm{det}}(v_N,\Lambda_d)\right|_T^2 - \left.\eta_{\mathrm{det}}(w_N,\Lambda_d)\right|_T^2 \right\vert \\
    &\qquad \leq \vert \eta_{\mathrm{det},T}(v_N,\Lambda_d)^2 - \eta_{\mathrm{det},T}(w_N,\Lambda_d)^2\vert  + \vert \eta_{\mathrm{det},\partial T}(v_N,\Lambda_d)^2 - \eta_{\mathrm{det},\partial T}(w_N,\Lambda_d)^2\vert \\
    &\qquad \leq 2 c_{\mathrm{inv}}c(\Lambda_d)\Bigl( \left.\eta_{\mathrm{det}}(v_N,\Lambda_d)\right|_T + \left.\eta_{\mathrm{det}}(w_N,\Lambda_d)\right|_T \Bigr)\Vert \grad(w_N-v_N)\Vert_{\pi_{\vartheta\rho}, T}.
  \end{align*}
 
\end{proof}

\begin{remark}
    \label{rem:properties:lipschitz_constant_det}
    Independent of the finiteness of the expansion of $a$, the constant $c(\Lambda_d)$ is always finite due to the recursive structure of the Hermite polynomials.
    However, $c(\Lambda_d)$ deteriorates if $\Lambda_d$ is enlarged, as Remark~\ref{rem:properties:doubly_orthogonal_polynomials} implies $\Vert P_\mu\zeta_{\vartheta\rho}\Vert_{\pi_0}\to\infty$ when $\vert\mu\vert\to\infty$.
    This is a direct consequence of the unboundedness of $a$ and hence the required introduction of adapted function spaces.
\end{remark}

In~\cite[Lemma 5.5]{EigelGittelson2014convergence} the authors assume a bound on the operator norm of the multiplication by $a$ instead of $a\in\mathcal{V}(\Lambda_d; \mathcal{T}, p)$, which would imply
\begin{equation}
  \label{eq:properties:bounded_multiplication}
  K_\beta = \sum_{\mu\in\mathcal{F}}\sum_{\alpha\in\mathcal{F}}\Vert a_\alpha\Vert_{L^\infty(D)} \, \vert\tau_{\alpha\beta\mu}\vert \leq c
  \qquad\mbox{for all }\beta\in\Lambda_d
\end{equation}
with a constant $c>0$ independent of $\beta$ for the general non-affine case.
The constant~\eqref{eq:properties:lipschitz_det:proof:constant} demonstrates that a similar bound is not possible for an unbounded $a$, since no decay in the coefficients $a_\alpha$ can counteract the deterioration of $\Vert P_\mu \zeta_{\vartheta\rho}\Vert_{\pi_0}$ as $\Lambda_d$ is enlarged.
For bounded $a$, however, well-posedness of~\eqref{eq:introduction:darcy} is given without the need for adapted function spaces and thus $\Vert P_\mu \zeta_{\vartheta\rho}\Vert_{\pi_0}=1$ independent of $\mu$.
As a consequence, a restriction of Theorem~\ref{thm:properties:lipschitz_det} to a bounded $a$ yields a generalization of~\cite[Lemma 5.5]{EigelGittelson2014convergence} to arbitrary bounded diffusion coefficients.

\begin{corollary}%
  \label{cor:properties:lipschitz_det_simple}
  If there exist constants $\check{c},\hat{c}>0$ such that $\check{c} \leq a(x,y) \leq \hat{c}$ uniformly for all $x\in D$ and almost all $y\in\Gamma$, then
  \begin{equation*}
    \bigl\vert \left.\eta_{\mathrm{det}}(v_N,\Lambda_d)\right|_T - \left.\eta_{\mathrm{det}}(w_N,\Lambda_d)\right|_T\bigr\vert
    \leq \sqrt{2}\, c_{\mathrm{inv}}\hat{c}\, \Vert \grad(v_N-w_N)\Vert_{\pi_{0},T},
  \end{equation*}
  where $c_{\mathrm{inv}}$ is the constant in the inverse estimates~\eqref{eq:setting:inverse_estimates}.
\end{corollary}

\begin{proof}%
  By the boundedness of $a$ we have ellipticity of~\eqref{eq:setting:B} and can choose $\vartheta=0$, which implies $\zeta_{\vartheta\rho}=1$ and thus $\pi_{\vartheta\rho}=\pi_0$.
  Similar to the proof of Theorem~\ref{thm:properties:lipschitz_det} we can use the third binomial formula and the triangle inequality to get
  \begin{equation}%
    \label{eq:properties:lipschitz_det_simple:proof:a}
    \begin{split}
      &\Bigl\vert \left.\eta_{\mathrm{det}}(v_N,\Lambda_d)\right|_T - \left.\eta_{\mathrm{det}}(w_N,\Lambda_d)\right|_T \Bigr\vert \Bigl( \left.\eta_{\mathrm{det}}(v_N,\Lambda_d)\right|_T + \left.\eta_{\mathrm{det}}(w_N,\Lambda_d)\right|_T \Bigr) \\
      &\qquad \leq \vert \eta_{\mathrm{det},T}(v_N,\Lambda_d)^2 - \eta_{\mathrm{det},T}(w_N,\Lambda_d)^2\vert  + \vert \eta_{\mathrm{det},\partial T}(v_N,\Lambda_d)^2 - \eta_{\mathrm{det},\partial T}(w_N,\Lambda_d)^2\vert \\
    \end{split}
  \end{equation}
  Using a combination of the third binomial formula, the orthonormality of ${P_\mu}$ and the second inverse estimate in~\eqref{eq:setting:inverse_estimates} allows us to estimate the first term of the right hand side of~\eqref{eq:properties:lipschitz_det_simple:proof:a}
  \begin{align}%
    \label{eq:properties:lipschitz_det_simple:proof:b}
    &\vert \eta_{\mathrm{det},T}(v_N,\Lambda_d)^2 - \eta_{\mathrm{det},T}(w_N,\Lambda_d)^2\vert \notag\\
    &\qquad\leq \Bigl\vert \int_{\Gamma}\int_{T} \Bigl( h_T \sum_{\mu\in\Lambda_d\cup\partial\Lambda_d} \ddiv r_\mu(w_N-v_N) P_\mu \Bigr) \Bigl( h_T\sum_{\mu\in\Lambda_d} (2f\delta_{\mu 0} - \ddiv r_\mu(v_N+w_N)) P_\mu\Bigr)\,\mathrm{d}x  \,\mathrm{d}\pi_{0}(y) \Bigr\vert \notag\\
    &\qquad\leq h_T \Vert \ddiv (a \grad(w_N-v_N))\Vert_{\pi_0, T} \Bigl( \eta_{\mathrm{det,T}}(v_N, \Lambda_d) + \eta_{\mathrm{det,T}}(w_N, \Lambda_d) \Bigr)\notag\\ 
    &\qquad\leq \hat{c}\,c_{\mathrm{inv}}\, \Vert \grad(w_N-v_N)\Vert_{\pi_0, T} \Bigl( \eta_{\mathrm{det,T}}(v_N, \Lambda_d) + \eta_{\mathrm{det,T}}(w_N, \Lambda_d) \Bigr).
  \end{align}
  Following the same arguments, but using the first inverse estimate in~\eqref{eq:setting:inverse_estimates} instead, provides an estimate for the second term on the right hand side of~\eqref{eq:properties:lipschitz_det_simple:proof:a}
  \begin{equation}%
    \label{eq:properties:lipschitz_det_simple:proof:c}
    \begin{split}
      &\vert \eta_{\mathrm{det},\partial T}(v_N,\Lambda_d)^2 - \eta_{\mathrm{det},\partial T}(w_N,\Lambda_d)^2\vert\\
      &\qquad\leq \hat{c}\,c_{\mathrm{inv}}\, \Vert \grad(w_N-v_N)\Vert_{\pi_0, T} \Bigl( \eta_{\mathrm{det,\partial T}}(v_N, \Lambda_d) + \eta_{\mathrm{det,\partial T}}(w_N, \Lambda_d) \Bigr).
    \end{split}
  \end{equation}
  The claim then follows by
  \begin{equation*}
    \eta_{\mathrm{det,T}}(v_N, \Lambda_d) + \eta_{\mathrm{det,\partial T}}(v_N, \Lambda_d)
    \leq \sqrt{2} \left.\eta_{\mathrm{det}}(v_N, \Lambda_d)\right\vert_T
    \qquad\mbox{for all }v_N\in\mathcal{V}_N.
  \end{equation*}
\end{proof}

\begin{remark}
    In the proof of Corollary~\ref{cor:convergence:convergence_bounded_a} we make use of $a\in\mathcal{V}_N(\Lambda_{\hat d}; \mathcal{T}, p)$ to apply the inverse estimates~\eqref{eq:setting:inverse_estimates}.
    The same result with a slightly different constant can be obtained for general bounded $a$ if we assume bounds for the diffusion coefficient as in~\eqref{eq:properties:bounded_multiplication} and its gradients, i.e.,
    \begin{equation*}
        \Vert h_T \grad a_\alpha \Vert_{L^\infty(D)} 
        \leq c \, \Vert a_\alpha \Vert_{L^\infty(D)}
    \end{equation*}
    for all $\alpha\in\Lambda_{\hat{d}}$ with some constant $c>0$ depending on the stochastic index set and the finite element mesh.
    In this case, Corollary~\ref{cor:convergence:convergence_bounded_a} provides a direct generalization of~\cite[Lemma 5.5]{EigelGittelson2014convergence} to arbitrary uniformly bounded diffusion coefficients.
\end{remark}

\begin{theorem}[continuity of $\eta_{\mathrm{det}}$ in the third component]%
  \label{thm:properties:continuity_det}
  Let $0\in\Lambda_d\subset\hat\Lambda\subset\mathbb{N}_0^{\hat M}$ be arbitrary sets, $\Delta = \hat\Lambda\cap\partial\Lambda_d$ and $w_N\in\mathcal{V}_N=\mathcal{V}_N(\Lambda_d;\mathcal{T},p)$.
  Then there exists a constant $\tilde{c}_{\mathrm{det}}>0$ such that
  \begin{equation*}
    \eta_{\mathrm{det}}(v_N,\mathcal{T},\hat\Lambda\setminus\Lambda_d)
    \leq \tilde{c}_{\mathrm{det}} \, \eta_{\mathrm{sto}}(v_N, \Delta).
  \end{equation*}
\end{theorem}

\begin{proof}%
  Since $0\not\in\Delta$, it holds
  \begin{align*}%
    \eta_{\mathrm{det}, T}(v_N, \Delta)^2
    &= h_T^2 \Vert \sum_{\mu\in\Delta} \ddiv r_\mu(v_N) P_\mu \zeta_{\vartheta\rho} \Vert_{\pi_0, T}^2
    = h_T^2 \sum_{\nu\in\Delta} c_\nu^2 \Vert \sum_{\mu\in\Delta} \ddiv r_\mu(v_N) z_{\mu\nu} \Vert_{T}^2\\
    &\leq c_{\mathrm{inv}}^2 \sum_{\nu\in\Delta} c_\nu^2 \Vert \sum_{\mu\in\Delta} r_\mu(v_N) z_{\mu\nu} \Vert_{T}^2
    = c_{\mathrm{inv}}^2 \Vert \sum_{\mu\in\Delta} r_\mu(v_N) P_\mu \zeta_{\vartheta\rho} \Vert_{\pi_0, T}^2,
  \end{align*}
  where we used Remark~\ref{rem:properties:doubly_orthogonal_polynomials},~\eqref{eq:properties:lipschitz_det:proof:a} and the second inverse estimate in~\eqref{eq:setting:inverse_estimates}.
  Similarily, using~\eqref{eq:properties:lipschitz_det:proof:b} and the first inverse estimate of~\eqref{eq:setting:inverse_estimates}, we get
  \begin{equation*}
    \eta_{\mathrm{det}, \partial T}(v_N, \Delta)^2
    \leq c_{\mathrm{inv}}^2 \Vert \sum_{\mu\in\Delta} r_\mu(v_N) P_\mu \zeta_{\vartheta\rho} \Vert_{\pi_0, T}^2,
  \end{equation*}
  We conclude the proof with the following estimate
  \begin{align*}%
    \eta_{\mathrm{det}}(v_N, \mathcal{T}, \hat{\Lambda}\setminus\Lambda_d)^2
    &= \sum_{T\in \mathcal{T}} \Bigl(\eta_{\mathrm{det}, T}(v_N, \Delta)^2 + \eta_{\mathrm{det}, \partial T}(v_N, \Delta)^2 \Bigr)\\
    &\leq 2c_{\mathrm{inv}}^2 \sum_{T\in \mathcal{T}} \Vert \sum_{\mu\in\Delta} r_\mu(v_N)P_\mu\zeta_{\vartheta\rho}\Vert_{\pi_0, T}^2\\
    &= 2c_{\mathrm{inv}}^2 \eta_{\mathrm{sto}}(v_N,\Delta)^2,
  \end{align*}
  where we used that $r_\mu(v_N)=0$ for any $\mu\in\hat\Lambda\setminus(\Lambda_d\cup\partial\Lambda_d)$.
  
\end{proof}

\begin{remark}%
  For the special case $\hat\Lambda\subseteq\Lambda_d\cup\partial\Lambda_d$, we have $\Delta=\hat\Lambda\cap\partial\Lambda_d=\hat\Lambda\setminus\Lambda_d$ and the inequality in Theorem~\ref{thm:properties:continuity_det} simplifies to
  \begin{equation*}
    \eta_{\mathrm{det}}(v_N,\mathcal{T},\Delta)
    \leq \tilde{c}_{\mathrm{det}} \, \eta_{\mathrm{sto}}(v_N, \Delta).
  \end{equation*}
\end{remark}

\subsection{Stochastic estimator contribution}%
\label{sec:properties:sto}

In this section we establish that the stochastic estimator contribution is Lipschitz continuous in the first component. We also introduce the quasi additivity of $\eta_{\mathrm{sto}}$ in the stochastic index set, which visualizes one of the key differences between a lognormal and a bounded affine diffusion coefficient.

\begin{theorem}[Lipschitz continuity of $\eta_{\mathrm{sto}}$ in the first component]%
  \label{thm:properties:lipschitz_sto}
  For any $v_N,w_N\in\mathcal{V}_N=\mathcal{V}_N(\Lambda_d;\mathcal{T},p)$ there exists a constant $c_{\mathrm{sto}}>0$ depending only on the boundary of the active set $\partial\Lambda_d$ such that
  \begin{equation*}
    \vert \eta_{\mathrm{sto}}(v_N,\partial\Lambda_d) - \eta_{\mathrm{sto}}(w_N,\partial\Lambda_d)\vert
    \leq c_{\mathrm{sto}} \Vert \grad(v_N-w_N)\Vert_{\pi_{\vartheta\rho},D}.
  \end{equation*}
\end{theorem}

\begin{proof}%
  Similar to the proof of Theorem~\ref{thm:properties:lipschitz_det}, we establish some intermediate results before proving the claim.
  First we note that Remark~\ref{rem:properties:doubly_orthogonal_polynomials} and the definition of $\eta_{\mathrm{sto}}$~\eqref{eq:estimator:eta_sto} imply
  \begin{equation}%
    \label{eq:properties:lipschitz_sto:proof:a}
    \eta_{\mathrm{sto}}(v_N, \partial\Lambda_d)^2 = \sum_{\nu\in\partial\Lambda_d} c_\nu^2 \, \Vert \sum_{\mu\in\partial\Lambda_d} r_\mu(v_N)z_{\mu\nu}\Vert_D^2
    \qquad\mbox{for all }v_N\in \mathcal{V}_N.
  \end{equation}
  Secondly, with $s_\nu := \Vert \sum_{\mu\in\partial\Lambda_d} r_\mu(v_N)z_{\mu\nu}\Vert_D + \Vert \sum_{\mu\in\partial\Lambda_d} r_\mu(w_N)z_{\mu\nu}\Vert_D$ it holds
  \begin{equation}%
    \label{eq:properties:lipschitz_sto:proof:b}
    \begin{split}
      \sum_{\nu\in\partial\Lambda_d} c_\nu^2 z_{\mu\nu} s_\nu
      &\leq \Bigl( \sum_{\nu\in\partial\Lambda_d} c_\nu^2 z_{\mu\nu}^2 \Bigr)^{1/2} \Bigl( \sum_{\nu\in\partial\Lambda_d} c_\nu^2 s_\nu^2 \Bigr)^{1/2}\\
      &\leq 2 \Vert P_\mu \zeta_{\vartheta\rho} \Vert_{\pi_0} \Bigl( \eta_{\mathrm{sto}}(v_N, \partial\Lambda_d) + \eta_{\mathrm{sto}}(w_N, \partial\Lambda_d) \Bigr).
    \end{split}
  \end{equation}
  We can now prove the claim by using~\eqref{eq:properties:lipschitz_sto:proof:a}--~\eqref{eq:properties:lipschitz_sto:proof:b}, the third binomial formula as well as the Cauchy-Schwarz, the triangle and the inverse triangle inequalities
  \begin{align*}%
    &\Bigl( \eta_{\mathrm{sto}}(v_N, \partial\Lambda_d) + \eta_{\mathrm{sto}}(w_N, \partial\Lambda_d) \Bigr) \vert \eta_{\mathrm{sto}}(v_N, \partial\Lambda_d) - \eta_{\mathrm{sto}}(w_N, \partial\Lambda_d) \vert\\
    &\qquad = \Bigl\vert \sum_{\nu\in\partial\Lambda_d} c_\nu^2 \, \Bigl( \Vert \sum_{\mu\in\partial\Lambda_d} r_\mu(v_N)z_{\mu\nu} \Vert_D^2 - \Vert \sum_{\mu\in\partial\Lambda_d} r_\mu(w_N)z_{\mu\nu} \Vert_D^2\Bigr)  \Bigr\vert\\
    &\qquad \leq \sum_{\nu\in\partial\Lambda_d} c_\nu^2s_\nu \, \Bigl\vert  \Vert \sum_{\mu\in\partial\Lambda_d} r_\mu(v_N)z_{\mu\nu} \Vert_D - \Vert \sum_{\mu\in\partial\Lambda_d} r_\mu(w_N)z_{\mu\nu} \Vert_D \Bigr\vert\\
    &\qquad \leq \sum_{\beta\in\Lambda_d} \Vert \grad(v_N-w_N)_\beta \Vert_D \Bigl( \sum_{\alpha\in\Lambda_{d+\hat{d}-1}} \sum_{\mu\in\partial\Lambda_d} \Vert a_\alpha \Vert_{L^\infty(D)} \vert \tau_{\alpha\beta\mu}\vert \sum_{\nu\in\partial\Lambda_d} c_\nu^2 z_{\mu\nu} s_\nu \Bigr)\\
    &\qquad \leq 2 c(\partial\Lambda_d) \Bigl( \eta_{\mathrm{sto}}(v_N,\partial\Lambda_d) + \eta_{\mathrm{sto}}(w_N,\partial\Lambda_d) \Bigr) \Vert \grad(v_N-w_N) \Vert_{\pi_{\vartheta\rho}, D},
  \end{align*}
  for $a_\alpha = \sum_{k=1}^{J} \discretized{a}_N[k,\alpha]\varphi_k\in\mathcal{X}_p(\mathcal{T})$ and
  \begin{equation}%
    \label{eq:properties:lipschitz_sto:proof:constant}
    c(\partial\Lambda_d)^2
    = \sum_{\beta\in\Lambda_d} \Bigl( \sum_{\mu\in\partial\Lambda_d} \sum_{\alpha\in\Lambda_{d+\hat{d}-1}} \Vert a_\alpha \Vert_{L^\infty(D)} \vert \tau_{\alpha\beta\mu} \vert \Vert P_\mu\zeta_{\vartheta\rho}\Vert_{\pi_0} \Bigr)^2.
  \end{equation}
  
\end{proof}

\begin{remark}%
  The constants~\eqref{eq:properties:lipschitz_det:proof:constant} and~\eqref{eq:properties:lipschitz_sto:proof:constant} are the same except for the sum over the index $\mu$.
  Since $\vert\partial\Lambda_d\vert<\infty$ due to $\vert\Lambda_{\hat d}\vert < \infty$, the constant in Theorem~\ref{thm:properties:lipschitz_sto} is qualitatively the same as in Theorem~\ref{thm:properties:lipschitz_det}.
  Note that the restriction of $a\in\mathcal{V}(\Lambda_{\hat d}; \mathcal{T}, p)$ for some finite $\Lambda_{\hat d}\subset \mathbb{N}_0^{\hat M}$ is essential for the Lipschitz continuity of $\eta_{\mathrm{sto}}$, as $\partial\Lambda_d$ and thus $c(\partial\Lambda_d)$ would be unbounded otherwise.
\end{remark}

Similar to Corrolary~\ref{cor:properties:lipschitz_det_simple}, a uniformly bounded and positive diffusion coefficient $a$ leads to a simplified Lipschitz constant.

\begin{corollary}%
  \label{cor:properties:lipschitz_sto_simple}
  If there exist constants $\check{c},\hat{c}>0$ such that $\check{c} \leq a(x,y) \leq \hat{c}$ uniformly for all $x\in D$ and almost all $y\in\Gamma$, then
  \begin{equation*}
    \vert \eta_{\mathrm{sto}}(v_N,\partial\Lambda_d) - \eta_{\mathrm{sto}}(w_N,\partial\Lambda_d)\vert
    \leq \hat{c}\, \Vert \grad(v_N-w_N)\Vert_{\pi_0,D}
    \qquad\mbox{for all }v_N,w_N\in\mathcal{V}_N.
  \end{equation*}
\end{corollary}

\begin{proof}%
  The boundedness of $a$ implies boundedness and ellipticity of the bilinear form~\eqref{eq:setting:B}.
  Hence we can choose $\vartheta=0$, i.e.\ $\zeta_{\vartheta\rho} \equiv 1$ and thus $\pi_{\vartheta\rho}=\pi_0$.
  The third binomial formula, the triangle inequality and the orthonormality of $\{ P_\mu \}$ now imply
  \begin{equation*}%
    \begin{split}
      &\Bigl(\eta_{\mathrm{sto}}(v_N,\partial\Lambda_d) + \eta_{\mathrm{sto}}(w_N,\partial\Lambda_d)\Bigr) \vert \eta_{\mathrm{sto}}(v_N,\partial\Lambda_d) - \eta_{\mathrm{sto}}(w_N,\partial\Lambda_d) \vert\\
      &\qquad=\vert \eta_{\mathrm{sto}}(v_N,\partial\Lambda_d)^2 - \eta_{\mathrm{sto}}(w_N,\partial\Lambda_d)^2 \vert\\
      &\qquad= \Bigl\vert \int_{\Gamma} \int_{D} \Bigl( \sum_{\mu\in\partial\Lambda_d} r_\mu(v_N-w_N)P_\mu \Bigr) \Bigl( \sum_{\mu\in\partial\Lambda_d} r_\mu(v_N+w_N)P_\mu \Bigr)\,\mathrm{d}x \,\mathrm{d}\pi_{0}(y) \Bigr\vert \\
      &\qquad= \Bigl\vert \int_{\Gamma} \int_{D} \Bigl( \sum_{\mu\in\Lambda_d\cup\partial\Lambda_d} r_\mu(v_N-w_N)P_\mu \Bigr) \Bigl( \sum_{\mu\in\partial\Lambda_d} r_\mu(v_N+w_N)P_\mu \Bigr)\,\mathrm{d}x \,\mathrm{d}\pi_{0}(y) \Bigr\vert \\
      &\qquad= \Bigl\vert \int_{\Gamma} \int_{D} a\grad(v_N-w_N) \Bigl( \sum_{\mu\in\partial\Lambda_d} r_\mu(v_N+w_N)P_\mu \Bigr)\,\mathrm{d}x \,\mathrm{d}\pi_{0}(y) \Bigr\vert \\
      &\qquad\leq \hat{c}\,\Vert \grad(v_N-w_N) \Vert_{\pi_{0},D} \,\Vert \sum_{\mu\in\partial\Lambda_d} r_\mu(v_N+w_N)P_\mu \Vert_{\pi_{0},D} \\
      &\qquad\leq \hat{c}\,\Vert \grad(v_N-w_N) \Vert_{\pi_{0},D} \Bigl( \eta_{\mathrm{sto}}(v_N,\partial\Lambda_d) + \eta_{\mathrm{sto}}(w_N,\partial\Lambda_d) \Bigr).
    \end{split}
  \end{equation*}
\end{proof}

We note that the Lipschitz continuity of $\eta_{\mathrm{sto}}$ for the affine field $\gamma$ was established in~\cite[Lemma 4.5]{EigelGittelson2014convergence}, which holds with the same Lipschitz constant.
Since $\gamma$ is a special case of a bounded positive diffusion field, Corollary~\ref{cor:properties:lipschitz_sto_simple} can be seen as a generalization of~\cite{EigelGittelson2014convergence}.

As the regularization parameter $\vartheta\in(0,1)$ influences the deviation of $\pi_{\vartheta\rho}$ from $\pi_0$, it is possible to show that $\eta_{\mathrm{sto}}$ is almost additive in the second argument if $\vartheta$ is chosen small enough.

\begin{theorem}[quasi additivity of $\eta_{\mathrm{sto}}$ in the second component]%
  \label{thm:properties:additivity_sto}
  For any $\varepsilon>0$, there exists $\vartheta_\varepsilon\in(0,1)$ such that for any $\vartheta \leq \vartheta_\varepsilon$ and any $\Delta\subseteq\partial\Lambda_d$
  \begin{equation*}
    \eta_{\mathrm{sto}}(w_N,\partial\Lambda_d\setminus\Delta)^2
    \leq \eta_{\mathrm{sto}}(w_N,\partial\Lambda_d)^2 - \eta_{\mathrm{sto}}(w_N,\Delta)^2 + \varepsilon.
  \end{equation*}
\end{theorem}

\begin{proof}%
  Let $g_\Delta(w_N)$ and $g_{\partial\Lambda_d\setminus\Delta}(w_N)$ be given respectively by
  \begin{equation*}
    g_\Delta(w_N) 
    = \sum_{\mu\in\Delta} r_\mu(w_N)P_\mu
    \qquad\mbox{and}\qquad
    g_{\partial\Lambda_d\setminus\Delta}(w_N)
    = \sum_{\mu\in\partial\Lambda_d\setminus\Delta} r_\mu(w_N)P_\mu.
  \end{equation*}
  Since $g_\Delta(w_N) \perp_{\pi_{\vartheta\rho},D} g_{\partial\Lambda_d\setminus\Delta}(w_N)$, the binomial formula and the Cauchy-Schwarz inequality yield
  \begin{equation*}%
    \begin{split}
      \eta_{\mathrm{sto}}(w_N,\partial\Lambda)^2
      &= \eta_{\mathrm{sto}}(w_N,\partial\Lambda_d\setminus\Delta)^2 + \eta_{\mathrm{sto}}(w_N,\Delta)^2 + 2\, \langle g_\Delta(w_N)g_{\partial\Lambda_d\setminus\Delta}(w_N),\; \zeta_{\vartheta\rho}^2-\zeta_{\vartheta\rho}\rangle_{\pi_0, D}\\
      &\geq \eta_{\mathrm{sto}}(w_N,\partial\Lambda_d\setminus\Delta)^2 + \eta_{\mathrm{sto}}(w_N,\Delta)^2 - 2\, \Vert g_\Delta(w_N)g_{\partial\Lambda_d\setminus\Delta}(w_N)\Vert_{\pi_0,D} \Vert \zeta_{\vartheta\rho}^2-\zeta_{\vartheta\rho}\Vert_{\pi_0}.
    \end{split}
  \end{equation*}
  By Lemma~\ref{lem:properties:weighted_gaussian_measure}, $\zeta_{\vartheta\rho}^{\alpha}\pi_0$ is proportional to a Gaussian probability density for any $\alpha>0$ as long as $\vartheta\rho \leq \rho_\alpha$.
  In particular, the normalization constant reads
  \begin{equation*}
    c_\alpha(\vartheta\rho) = \prod_{m=1}^{\hat M} c_{\alpha,m}(\vartheta\rho)
    \qquad\mbox{with}\qquad
    c_{\alpha,m}(\vartheta\rho) = \sigma_m(\vartheta\rho)^{\alpha-1}\sqrt{\alpha + (1-\alpha)\sigma_m(\vartheta\rho)^2}.
  \end{equation*}
  Since $\sigma_m(\vartheta\rho)^\alpha = \exp(\alpha\vartheta\rho\Vert\gamma_m\Vert_{L^\infty(D)}) \to 1$ as $\vartheta\to0$ for any $\alpha\in\mathbb{R}$, we get $c_\alpha(\vartheta\rho)\to1$.
  This implies
  \begin{equation*}
    0
    \leq \Vert \zeta_{\vartheta\rho}^2-\zeta_{\vartheta\rho}\Vert_{\pi_0}^2
    = \int_{\Gamma} \zeta_{\vartheta\rho}^4 \,\mathrm{d}\pi_0 + \int_{\Gamma} \zeta_{\vartheta\rho}^2 \,\mathrm{d}\pi_0 - 2\int_{\Gamma} \zeta_{\vartheta\rho}^3 \,\mathrm{d}\pi_0
    \ \xrightarrow{\vartheta\to0}\ 0.
  \end{equation*}
  Since $\Vert g_\Delta(w_N)g_{\partial\Lambda_d\setminus\Delta}(w_N)\Vert_{\pi_0,D}$ is independent of $\vartheta$, Lemma~\ref{lem:properties:weighted_gaussian_measure} yields that there exists $0 < \vartheta_\varepsilon \leq \min\{\rho_{\alpha}\ \vert \ \alpha=2,3,4\}$ such that
  \begin{equation*}
    \Vert \zeta_{\vartheta\rho}^2-\zeta_{\vartheta\rho}\Vert_{\pi_0}
    \leq \frac{1}{2}\varepsilon \Vert g_\Delta(w_N)g_{\partial\Lambda_d\setminus\Delta}(w_N)\Vert_{\pi_0,D}^{-1},
  \end{equation*}
  for any $\vartheta < \vartheta_\varepsilon$, which proves the claim.
\end{proof}

\section{Quasi-error reduction by the adaptive algorithm}%
\label{sec:convergence}

With the properties established in the previous section, this section proves the reduction of the quasi-error~\eqref{eq:introduction:convergence} in each iteration of the adaptive Algorithm~\ref{alg:estimator:adaptive} as the main result of this work.
As depicted in Figure~\ref{fig:convergence:schematic}, it is first required to establish an estimate that relates the estimator contributions on one level to similar quantities of the previous level.

\begin{lemma}%
  \label{lem:convergence:main_lemma}
  For any non-empty sets $0\in\Lambda_d\subset\hat\Lambda\subset\mathbb{N}_0^{\hat M}$ and triangulations $\mathcal{T},\hat{\mathcal{T}}$, where $\hat{\mathcal{T}}$ is a one-level refinement of $\mathcal{T}$, let $\mathcal{M}=\mathcal{T}\setminus(\hat{\mathcal{T}}\cap\mathcal{T})$ be the set of triangles marked for refinement and $\Delta=\partial\Lambda_d\cap\hat\Lambda$ the set of added stochastic indices.
  Then it holds for any $w_N\in\mathcal{V}_N=\mathcal{V}_N(\Lambda_d;\mathcal{T},p)$, $\hat w_N\in\hat{\mathcal{V}}_N=\hat{\mathcal{V}}_N(\hat\Lambda;\hat{\mathcal{T}},p)$ and $\varepsilon_{\mathrm{det}}, \varepsilon_{\mathrm{sto}} > 0$, $\tau\geq 0$ that
  \begin{equation*}
    \begin{split}
      &\eta_{\mathrm{det}}(\hat w_N,\hat{\mathcal{T}},\hat\Lambda)^2 + \tau \eta_{\mathrm{sto}}(\hat w_N, \partial\hat\Lambda)^2\\
      &\qquad\leq (1+\varepsilon_{\mathrm{det}})\Bigl(\eta_{\mathrm{det}}(w_N,\mathcal{T},\Lambda_d)^2 - \lambda\, \eta_{\mathrm{det}}(w_N,\mathcal{M},\Lambda_d)^2\Bigr)\\
      &\qquad\quad+ (1+\varepsilon_{\mathrm{sto}})\tau \, \eta_{\mathrm{sto}}(w_N,\partial\Lambda_d\setminus\Delta)^2 + 3(1+\varepsilon_{\mathrm{det}})c_{\mathrm{inv}}^2 \, \eta_{\mathrm{sto}}(w_N,\Delta)^2\\
      &\qquad\quad+ \Bigl((1+\varepsilon_{\mathrm{det}}^{-1})c_{\mathrm{det}}^2 + (1+\varepsilon_{\mathrm{sto}}^{-1})c_{\mathrm{sto}}^2\tau\Bigr) \, \Vert \grad(w_N-\hat w_N))\Vert_{\pi_{\rho},D}^2,
    \end{split}
  \end{equation*}
  with $\lambda = 1 - 2^{-1/2}$ and $c_{\mathrm{sto}}$, $c_{\mathrm{det}}$ from Theorem~\ref{thm:properties:lipschitz_det} and Theorem~\ref{thm:properties:lipschitz_sto}, respectively.
\end{lemma}

\begin{proof}%
  By Theorem~\ref{thm:properties:lipschitz_det} we have
  \begin{equation*}
    \begin{split}
      \eta_{\mathrm{det}}(\hat w_N,\hat{\mathcal{T}},\hat\Lambda)^2
      &\leq \sum_{\hat T\in\hat{\mathcal{T}}} \Bigl(\eta_{\mathrm{det},\hat T}(w_N,\hat\Lambda) + \vert \eta_{\mathrm{det},\hat T}(\hat w_N,\hat\Lambda) - \eta_{\mathrm{det},\hat T}(w_N,\hat\Lambda)\vert \Bigr)^2\\
      &\leq \sum_{\hat T\in\hat{\mathcal{T}}} \Bigl(\eta_{\mathrm{det},\hat T}(w_N,\hat\Lambda) + c_{\mathrm{det}}\Vert\grad(w_N-\hat w_N)\Vert_{\pi_{\vartheta\rho},\hat T} \Bigr)^2.
    \end{split}
  \end{equation*}
  Using Young's inequality for the mixed terms of the last estimate yields for any $\varepsilon_{\mathrm{det}}>0$
  \begin{equation*}
    \begin{split}
      &2 c_{\mathrm{det}}\eta_{\mathrm{det},\hat T}(w_N,\hat\Lambda) \Vert\grad(w_N-\hat w_N)\Vert_{\pi_{\vartheta\rho},\hat T}\\
      &\qquad\leq \varepsilon_{\mathrm{det}} \eta_{\mathrm{det},\hat T}(w_N,\hat\Lambda)^2 + \varepsilon_{\mathrm{det}}^{-1} c_{\mathrm{det}}^2 \Vert\grad(w_N-\hat w_N)\Vert_{\pi_{\vartheta\rho},\hat T}^2,
    \end{split}
  \end{equation*}
  which implies
  \begin{equation*}
    \eta_{\mathrm{det}}(\hat w_N,\hat{\mathcal{T}},\hat\Lambda)^2
    \leq (1+\varepsilon_{\mathrm{det}}) \eta_{\mathrm{det}}(w_N,\hat{\mathcal{T}},\hat\Lambda)^2 + (1+\varepsilon_{\mathrm{det}}^{-1}) c_{\mathrm{det}}^2 \Vert\grad(w_N-\hat w_N)\Vert_{\pi_{\vartheta\rho},D}^2.
  \end{equation*}
  Applying Theorem~\ref{thm:properties:continuity_det} then gives
  \begin{equation*}
    \begin{split}
      \eta_{\mathrm{det}}(w_N,\hat{\mathcal{T}},\hat\Lambda)^2
      &\leq \eta_{\mathrm{det}}(w_N,\hat{\mathcal{T}},\Lambda_d)^2 + \eta_{\mathrm{det}}(w_N,\hat{\mathcal{T}},\hat\Lambda\setminus\Lambda_d)^2 \\
      &\leq \eta_{\mathrm{det}}(w_N,\hat{\mathcal{T}},\Lambda_d)^2 + 3 c_{\mathrm{inv}}^2\eta_{\mathrm{sto}}(w_N,\partial\Lambda_d\cap\hat\Lambda)^2.
    \end{split}
  \end{equation*}
  Let $T\in\mathcal{M}\subset\mathcal{T}$ be a triangle marked for refinement and denote by $\hat{\mathcal{T}}(T) = \{ \hat T\in \hat{\mathcal{T}} \colon \hat T\subset T \}$ the set of all children of $T$ in $\hat{\mathcal{T}}$.
  Since $w_N$ is smooth on all edges $\hat E\in\operatorname{int}(T)$ it follows that $\llbracket r_\mu(w_N)\rrbracket_{\hat E}=0$ for all $\mu\in\Lambda_d$.
  Since we assume $D\subset\mathbb{R}^2$ and $\hat{\mathcal{T}}$ to be a one-level refinement of $\mathcal{T}$ obtained via newest-vertex bisection, there holds
  \begin{equation*}
    h_{\hat T}
    = \vert \hat T \vert^{1/2}
    \leq \Big( \frac{1}{2} \vert T\vert \Bigr)^{1/2}
    = 2^{-1/2} h_T
  \end{equation*}
  for any $\hat T\in \hat{\mathcal{T}}(T)$.
  We note that technically $h_{T}\approx\vert T \vert^{1/2}$ with equivalence constants induced by the shape regularity of $\mathcal{T}$, which we will ignore here to keep the notation as concise as possible.
  With $\lambda=1-2^{-1/2}$ we get
  \begin{equation*}
    \begin{split}
      \eta_{\mathrm{det}}(w_N,\hat{\mathcal{T}},\Lambda_d)^2
      &= \eta_{\mathrm{det}}(w_N,\hat{\mathcal{T}}\setminus\hat{\mathcal{T}}(\mathcal{M}),\Lambda)^2 + \eta_{\mathrm{det}}(w_N,\hat{\mathcal{T}}(\mathcal{M}),\Lambda)^2\\
      &\leq \eta_{\mathrm{det}}(w_N,\mathcal{T}\setminus\mathcal{M},\Lambda_d)^2 + 2^{-1/2} \, \eta_{\mathrm{det}}(w_N,\mathcal{M},\Lambda_d)^2\\
      &= \eta_{\mathrm{det}}(w_N,\mathcal{T},\Lambda_d)^2 - \lambda \, \eta_{\mathrm{det}}(w_N,\mathcal{M},\Lambda_d)^2.
    \end{split}
  \end{equation*}
  Combining the above estimates yields
  \begin{equation*}
    \begin{split}
      \eta_{\mathrm{det}}(\hat w_N, \hat{\mathcal{T}}, \hat\Lambda)^2
      &\leq (1+\varepsilon_{\mathrm{det}}) \Bigl( \eta_{\mathrm{det}}(w_N,\mathcal{T},\Lambda_d)^2 - \lambda \, \eta_{\mathrm{det}}(w_N,\mathcal{M},\Lambda_d)^2 \Bigr)\\
      &\qquad + 3(1+\varepsilon_{\mathrm{det}})c_{\mathrm{inv}}^2 \, \eta_{\mathrm{sto}}(w_N,\Delta)^2 + (1+\varepsilon_{\mathrm{det}}^{-1})c_{\mathrm{det}}^2 \, \Vert\grad(w_N-\hat w_N)\Vert_{\pi_{\vartheta\rho},D}^2.
    \end{split}
  \end{equation*}
  Similarly, Theorem~\ref{thm:properties:lipschitz_sto} and Young's inequality for any $\varepsilon_{\mathrm{sto}}>0$ leads to the estimate
  \begin{equation*}
    \begin{split}
      \eta_{\mathrm{sto}}(\hat w_N,\partial\hat\Lambda)^2
      &\leq \Bigl(\eta_{\mathrm{sto}}(w_N,\partial\hat\Lambda) + \vert \eta_{\mathrm{sto}}(\hat w_N,\partial\hat\Lambda) - \eta_{\mathrm{st}}(w_N,\partial\hat\Lambda)\vert \Bigr)^2\\
      &\leq \Bigl(\eta_{\mathrm{sto}}(w_N,\partial\hat\Lambda) + c_{\mathrm{sto}} \, \Vert\grad(w_N-\hat w_N)\Vert_{\pi_{\vartheta\rho},D} \Bigr)^2\\
      &\leq (1+\varepsilon_{\mathrm{sto}}) \, \eta_{\mathrm{sto}}(w_N,\partial\hat\Lambda)^2 + (1+\varepsilon_{\mathrm{sto}}^{-1})c_{\mathrm{sto}}^2 \, \Vert\grad(w_N-\hat w_N)\Vert_{\pi_{\vartheta\rho},D}^2.
    \end{split}
  \end{equation*}
  Note that $\Lambda_d\subset\hat\Lambda$ implies $\partial\Lambda_d\subset\hat\Lambda\cup\partial\hat\Lambda$ and thus $\Delta=\partial\Lambda_d\cap\hat\Lambda=\partial\Lambda_d\setminus\partial\hat\Lambda$.
  Since $r_\mu(w_N)=0$ for $\mu\not\in\Lambda_d\cup\partial\Lambda_d$, we get that $\eta_{\mathrm{sto}}(w_N,\partial\hat\Lambda\setminus\partial\Lambda_d)^2=0$, which yields
  \begin{equation*}
    \eta_{\mathrm{sto}}(\hat{w}_N,\partial\hat\Lambda)^2
    = \eta_{\mathrm{sto}}(\hat{w}_N,\partial\Lambda_d\cap\partial\hat\Lambda)^2
    = \eta_{\mathrm{sto}}(\hat{w}_N,\partial\Lambda_d\setminus(\partial\Lambda_d\cap\hat\Lambda))^2
    = \eta_{\mathrm{sto}}(\hat{w}_N,\partial\Lambda_d\setminus\Delta)^2.
  \end{equation*}
  Combining all the results above and estimating the norm by Lemma~\ref{lem:properties:weighted_gaussian_measure} concludes the proof.
\end{proof}

With Lemma~\ref{lem:convergence:main_lemma}, Lemma~\ref{lem:properties:weighted_gaussian_measure} and Theorem~\ref{thm:properties:additivity_sto} we can now prove reduction of the quasi error~\eqref{eq:introduction:convergence} on each level.

\begin{theorem}[quasi-error reduction]%
  \label{thm:convergence:convergence}
  Let $c_{\mathrm{eq}}>0$, $0 < \theta_{\mathrm{det}},\theta_{\mathrm{sto}}<1$ and let $u_\ell$, $\mathcal{T}_\ell$, $\mathcal{M}_\ell$, $\Lambda_\ell$, $\Delta_\ell$, $\eta_{\mathrm{det},\ell}$ and $\eta_{\mathrm{sto},\ell}$ denote a sequence of approximate solutions, triangulations, marked cells, stochastic indices, marked indices and error indicators, respectively, generated by the adaptive Algorithm~\ref{alg:estimator:adaptive}.
  Then there exist $0<\delta_\ell<1$, $\omega_\ell>0$, $\tau>0$ and a regularization threshold $0<\vartheta^*<1$, such that for any $\vartheta\leq\vartheta^*$ it holds
  \begin{equation*}%
    \Vert u-u_{\ell+1}\Vert_{B}^2 + \omega_{\ell} \, \eta_{\mathrm{det},\ell+1}^2 + \omega_{\ell}\tau \, \eta_{\mathrm{sto},\ell+1}^2
    \leq \delta_\ell\Bigl( \Vert u-u_{\ell}\Vert_{B}^2 + \omega_{\ell} \, \eta_{\mathrm{det},\ell}^2 + \omega_{\ell}\tau \, \eta_{\mathrm{sto},\ell}^2\Bigr).
  \end{equation*}
\end{theorem}

\begin{proof}%
  Let $e_\ell:=\Vert u-u_\ell\Vert_{B}$, $d_\ell:=\Vert u_\ell-u_{\ell+1}\Vert_{B}$ and $\tilde d_\ell:=\Vert\grad(u_\ell-u_{\ell+1})\Vert_{\pi_{\rho},D}$.
  With Galerkin orthogonality $e_{\ell+1}^2 = e_\ell^2 - d_\ell^2$ and Lemma~\ref{lem:convergence:main_lemma} it follows
  \begin{equation*}
    \begin{split}
      &e_{\ell+1}^2 + \omega \, \eta_{\mathrm{det},\ell+1}^2 + \omega\tau \, \eta_{\mathrm{sto},\ell+1}^2\\
      &\qquad\leq e_\ell^2 + \omega\Bigl((1+\varepsilon_{\mathrm{det}}^{-1})c_{\mathrm{det}}^2 + (1+\varepsilon_{\mathrm{sto}}^{-1})c_{\mathrm{sto}}^2\tau\Bigr)\tilde d_\ell^2 - d_\ell^2\\
      &\qquad\qquad+ \omega(1+\varepsilon_{\mathrm{det}}) \, \eta_{\mathrm{det},\ell}^2 - \omega(1+\varepsilon_{\mathrm{det}})\lambda \, \eta_{\mathrm{det},\ell}(u_\ell,\mathcal{M}_\ell,\Lambda_\ell)^2\\
      &\qquad\qquad+ \omega(1+\varepsilon_{\mathrm{sto}})\tau \, \eta_{\mathrm{sto},\ell}(u_\ell,\partial\Lambda_\ell\setminus\Delta_\ell)^2 + 3\omega(1+\varepsilon_{\mathrm{det}})c_{\mathrm{inv}}^2 \, \eta_{\mathrm{sto},\ell}(u_\ell,\Delta_\ell)^2.
    \end{split}
  \end{equation*}
  Let $\omega_\ell^* := \hat{c}_{\vartheta\rho}^{-1}\bigl( (1+\varepsilon_{\mathrm{det}}^{-1})c_{\mathrm{det}}^2 + (1+\varepsilon_{\mathrm{sto}}^{-1})c_{\mathrm{sto}}^2\tau \bigr)^{-1}$, where $\hat{c}_{\vartheta\rho}$ is the boundedness constant in~\eqref{eq:setting:boundedness_B}, and let $\omega_\ell=\omega_{\ell}^* d_\ell^2 \tilde d_\ell^{-2}$ such that the terms containing $d_\ell$ and $\tilde d_\ell$ cancel each other.
  Note that $\omega_\ell$ can always be chosen this way since $d_\ell>0$ implies $0 < \omega_\ell \leq \omega_\ell^*<\infty$.
  Next, we introduce the convex combination
  \begin{equation*}
    e_\ell^2 
    = (1-\alpha)e_\ell^2 + \alpha e_\ell^2
    \leq (1-\alpha)e_\ell^2 + \alpha c_{\mathrm{rel}}^2\eta_{\mathrm{det},\ell}^2 + \alpha c_{\mathrm{rel}}^2c_{\mathrm{eq}}^2\eta_{\mathrm{sto},\ell}^2
  \end{equation*}
  for any $\alpha\in(0,1)$, where $c_{\mathrm{rel}}$ is the reliability constant from Theorem~\ref{thm:estimator:reliability} and $c_{\mathrm{eq}}$ is the equilibration constant from~\eqref{eq:estimator:eta}.
  With this it follows
  \begin{equation*}
    \begin{split}
      &e_{\ell+1}^2 + \omega_{\ell} \, \eta_{\mathrm{det},\ell+1}^2 + \omega_{\ell}\tau \, \eta_{\mathrm{sto},\ell+1}^2\\
      &\qquad\leq (1-\alpha)e_\ell^2 + \Bigl(\alpha c_{\mathrm{rel}}^2 + \omega_{\ell}(1+\varepsilon_{\mathrm{det}})\Bigr)\eta_{\mathrm{det},\ell}^2\\
      &\qquad\qquad- \omega_{\ell}(1+\varepsilon_{\mathrm{det}})\lambda \, \eta_{\mathrm{det},\ell}(u_\ell,\mathcal{M}_\ell,\Lambda_\ell)^2 + \alpha c_{\mathrm{rel}}^2c_{\mathrm{eq}}^2 \, \eta_{\mathrm{sto},\ell}^2\\
      &\qquad\qquad+ \omega_{\ell}(1+\varepsilon_{\mathrm{sto}})\tau \, \eta_{\mathrm{sto},\ell}(u_\ell,\partial\Lambda_\ell\setminus\Delta_\ell)^2 + 3\omega_{\ell}(1+\varepsilon_{\mathrm{det}})c_{\mathrm{inv}}^2 \, \eta_{\mathrm{sto},\ell}(u_\ell,\Delta_\ell)^2.
  \end{split}
  \end{equation*}
  Next we need to distinguish between the different marking scenarios of Algorithm~\ref{alg:estimator:adaptive}.
  We first consider refinement of the spatial domain, i.e., $\eta_{\mathrm{det}, \ell}\geq c_{\mathrm{eq}}\eta_{\mathrm{sto},\ell}$, which implies $\Delta_\ell=\emptyset$ and
  \begin{equation*}
    \begin{split}
      &\alpha c_{\mathrm{rel}}^2 c_{\mathrm{eq}}^2 \, \eta_{\mathrm{sto},\ell}^2 + \omega_{\ell}(1+\varepsilon_{\mathrm{sto}})\tau \, \eta_{\mathrm{sto},\ell}(u_\ell,\partial\Lambda_\ell\setminus\Delta_\ell)^2 + 3\omega_{\ell}(1+\varepsilon_{\mathrm{det}})c_{\mathrm{inv}}^2 \, \eta_{\mathrm{sto},\ell}(u_\ell,\Delta_\ell)^2\\
      &\qquad= \omega_{\ell}\tau\Bigl(\alpha c_{\mathrm{rel}}^2 c_{\mathrm{eq}}^2\omega_{\ell}^{-1}\tau^{-1} + (1+\varepsilon_{\mathrm{sto}})\Bigr)\eta_{\mathrm{sto},\ell}^2\\
      &\qquad= \omega_{\ell}\tau\varepsilon_{\mathrm{sto}}(1+\beta_1) \, \eta_{\mathrm{sto},\ell}^2 + \omega_{\ell}\tau c_2 \, \eta_{\mathrm{sto},\ell}^2,
    \end{split}
  \end{equation*}
  for any $\beta_1 > 0$ and 
  \begin{equation*}
    c_2
    = c_2(\alpha,\varepsilon_{\mathrm{sto}},\tau,\beta_1)
    = \alpha c_{\mathrm{rel}}^2c_{\mathrm{eq}}^2\omega_{\ell}^{-1}\tau^{-1} + (1-\varepsilon_{\mathrm{sto}}\beta_1).
  \end{equation*}
  Moreover, by the D\"orfler criterion we have $\eta_{\mathrm{det},\ell}(u_\ell,\mathcal{M}_\ell,\Lambda_\ell)\geq\theta_{\mathrm{det}}\eta_{\mathrm{det},\ell}$ and since $\eta_{\mathrm{sto},\ell}\leq c_{\mathrm{eq}}^{-1}\eta_{\mathrm{det},\ell}$ we obtain
  \begin{equation*}
    \begin{split}
      &\Bigl(\alpha c_{\mathrm{rel}}^2 + \omega_{\ell}(1+\varepsilon_{\mathrm{det}})\Bigr)\eta_{\mathrm{det},\ell}^2 - \omega_{\ell}(1+\varepsilon_{\mathrm{det}})\lambda \, \eta_{\mathrm{det},\ell}(u_\ell,\mathcal{M}_\ell,\Lambda_\ell)^2 + \omega_{\ell}\tau\varepsilon_{\mathrm{sto}}(1+\beta_1) \, \eta_{\mathrm{sto},\ell}^2\\
      &\qquad\qquad\leq \omega_{\ell} c_1 \, \eta_{\mathrm{det},\ell}^2,
    \end{split}
  \end{equation*}
  for
  \begin{equation*}
    \begin{split}
      c_1 
      = c_1(\alpha,\varepsilon_{\mathrm{det}},\varepsilon_{\mathrm{sto}},\tau,\beta_1)
      = \alpha c_{\mathrm{rel}}^2\omega_{\ell}^{-1} + (1+\varepsilon_{\mathrm{det}})(1-\lambda\theta_{\mathrm{det}}^2)+\tau\varepsilon_{\mathrm{sto}}(1+\beta_1)c_{\mathrm{eq}}^{-2}.
    \end{split}
  \end{equation*}
  We thus have
  \begin{equation}%
    \label{eq:convergence:convergence:case_a}
    e_{\ell+1}^2 + \omega_{\ell} \, \eta_{\mathrm{det},\ell+1}^2 + \omega_{\ell}\tau \, \eta_{\mathrm{sto},\ell+1}^2
    \leq (1-\alpha)e_{\ell}^2 + \omega_{\ell} c_1 \, \eta_{\mathrm{det},\ell}^2 + \omega_{\ell}\tau c_2 \, \eta_{\mathrm{sto},\ell}^2.
  \end{equation}
  In the second case, when Algorithm~\ref{alg:estimator:adaptive} refines the stochastic space, we have $\eta_{\mathrm{det},\ell}<c_{\mathrm{eq}}\eta_{\mathrm{sto},\ell}$, which implies $\mathcal{M}_\ell=\emptyset$ and
  \begin{equation*}
    \begin{split}
      &\Bigl(\alpha c_{\mathrm{rel}}^2 + \omega_{\ell}(1+\varepsilon_{\mathrm{det}})\Bigr)\eta_{\mathrm{det},\ell}^2 - \omega_{\ell}(1+\varepsilon_{\mathrm{det}})\lambda \, \eta_{\mathrm{det},\ell}(u_\ell,\mathcal{M}_\ell,\Lambda_\ell)^2\\
      &\qquad= \omega_{\ell}\Bigl(\alpha c_{\mathrm{rel}}^2\omega_{\ell}^{-1} + (1+\varepsilon_{\mathrm{det}})\Bigr)\eta_{\mathrm{det},\ell}^2\\
      &\qquad= \omega_{\ell} c_3 \, \eta_{\mathrm{det},\ell}^2 + \omega_{\ell}\varepsilon_{\mathrm{det}}(1+\beta_2) \, \eta_{\mathrm{det},\ell}^2
    \end{split}
  \end{equation*}
  for any $\beta_2 > 0$ and
  \begin{equation*}
    c_3 = c_3(\alpha,\varepsilon_{\mathrm{det}},\beta_2)
    = \alpha c_{\mathrm{rel}}^2\omega_{\ell}^{-1} + (1-\varepsilon_{\mathrm{det}}\beta_2).
  \end{equation*}
  Again, by the D\"orfler criterion, it holds $\eta_{\mathrm{sto},\ell}(u_\ell,\Delta_\ell)\geq\theta_{\mathrm{sto}}\eta_{\mathrm{sto},\ell}$ and in combination with $\eta_{\mathrm{det},\ell}\leq c_{\mathrm{eq}}\eta_{\mathrm{sto},\ell}$ and Theorem~\ref{thm:properties:additivity_sto} we estimate
  \begin{equation*}
    \begin{split}
      &\alpha c_{\mathrm{rel}}^2c_{\mathrm{eq}}^2 \, \eta_{\mathrm{sto},\ell}^2 + \omega_{\ell}\varepsilon_{\mathrm{det}}(1+\beta_2) \, \eta_{\mathrm{det},\ell}^2\\
      &\quad+ \omega_{\ell}(1+\varepsilon_{\mathrm{sto}})\tau \, \eta_{\mathrm{sto},\ell}(u_\ell,\partial\Lambda_\ell\setminus\Delta_\ell)^2 + 3\omega_{\ell}(1+\varepsilon_{\mathrm{det}})c_{\mathrm{inv}}^2 \, \eta_{\mathrm{sto},\ell}(u_\ell,\Delta_\ell)^2\\
      &\qquad\qquad\leq \omega_{\ell}\tau c_4 \, \eta_{\mathrm{sto},\ell}^2,
    \end{split}
  \end{equation*}
  where we set 
  \begin{equation*}
    \begin{split}
      c_4 &= c_4(\alpha,\varepsilon_{\mathrm{det}},\varepsilon_{\mathrm{sto}},\tau,\beta_2,\vartheta)\\
          &= \alpha \tau^{-1}\check{c}_{\vartheta\rho}^{-1}c_{\mathrm{rel}}^2c_{\mathrm{eq}}^2 + \tau^{-1}\varepsilon_{\mathrm{det}}c_{\mathrm{eq}}^2(1+\beta_2) + (1+\varepsilon_{\mathrm{sto}})(1+\varepsilon_\vartheta)\\
      &\qquad- \theta_{\mathrm{sto}}^2\Bigl(1+\varepsilon_{\mathrm{sto}} - 3(1+\varepsilon_{\mathrm{det}})c_{\mathrm{inv}}^2\tau^{-1}\Bigr).
    \end{split}
  \end{equation*}
  Here, we set $0 < \varepsilon_\vartheta \leq \varepsilon_\vartheta^*\eta_{\mathrm{sto},\ell}^{-2}$, where $\varepsilon_\vartheta^*$ is the maximal $\varepsilon$ such that Theorem~\ref{thm:properties:additivity_sto} holds for $\vartheta$.
  Similar to~\eqref{eq:convergence:convergence:case_a}, this now yields the estimate
  \begin{equation}
    \label{eq:convergence:convergence:case_b}
    e_{\ell+1}^2 + \omega_{\ell} \, \eta_{\mathrm{det},\ell+1}^2 + \omega_{\ell}\tau \, \eta_{\mathrm{sto},\ell+1}^2
    \leq (1-\alpha)e_{\ell}^2 + \omega_{\ell} c_3 \, \eta_{\mathrm{det},\ell}^2 + \omega_{\ell}\tau c_4 \, \eta_{\mathrm{sto},\ell}^2.
  \end{equation}
  What remains is to choose the parameters $\alpha$, $\varepsilon_{\mathrm{det}}$, $\varepsilon_{\mathrm{sto}}$, $\tau$, $\beta_1$ $\beta_2$ and $\vartheta$ such that simultaneously $0 < c_1, \dots, c_4 < 1$.
  First we note that $c_1>0$ is trivially satisfied since $\lambda<1$ and thus $1-\lambda\theta_{\mathrm{det}}^2>0$ independent of the choice of $\theta_{\mathrm{det}}\in(0,1)$.
  With
  \begin{equation*}
    \begin{split}
      \varepsilon_{\mathrm{det}} < \frac{\lambda\theta_{\mathrm{det}}^2}{3(1-\lambda\theta_{\mathrm{det}}^2)},
      \quad
      \varepsilon_{\mathrm{sto}} < \frac{\lambda\theta_{\mathrm{det}}^2c_{\mathrm{eq}}^2}{3\tau(1+\beta_1)}
      \qquad\mbox{and}\qquad
      \alpha < \frac{\lambda\theta_{\mathrm{det}}^2\omega_{\ell}}{3c_{\mathrm{rel}}^2}
    \end{split}
  \end{equation*}
  we ensure that $c_1 < 1$.
  If additionally
  \begin{equation*}
    \frac{1}{2\varepsilon_{\mathrm{sto}}}<\beta_1<\frac{1}{\varepsilon_{\mathrm{sto}}}
    \qquad\mbox{and}\qquad
    \alpha < \frac{\omega_{\ell}\tau}{2c_{\mathrm{rel}}^2c_{\mathrm{eq}}^2},
  \end{equation*}
  we guarantee that $0 < \alpha c_{\mathrm{rel}}^2c_{\mathrm{eq}}^2\omega_{\ell}^{-1}\tau^{-1} < c_2 < 1$.
  To ensure that $0<c_4<1$, we set $\tau>3(1+\varepsilon_{\mathrm{det}})c_{\mathrm{inv}}^2$ such that $1-3(1+\varepsilon_{\mathrm{det}})c_{\mathrm{inv}}^2\tau^{-1}>0$.
  By Theorem~\ref{thm:properties:additivity_sto} there exist $\vartheta^*\in(0,1)$ such that
  \begin{equation*}
   0 
   < \varepsilon_\vartheta 
   < \frac{\theta_{\mathrm{sto}}^2(1-3c_{\mathrm{inv}}^2\tau^{-1})}{(2+\frac{3}{2c_{\mathrm{eq}}^2(1+\beta_2)}\theta_{\mathrm{sto}}^2c_{\mathrm{inv}}^2)}
   \qquad\mbox{for all }\vartheta<\vartheta^*.
  \end{equation*}
  Now we choose
  \begin{equation*}
    \varepsilon_{\mathrm{det}}<\frac{\varepsilon_\vartheta\tau}{2c_{\mathrm{eq}}^2(1+\beta_2)},
    \qquad
    \varepsilon_{\mathrm{sto}} < \frac{\theta_{\mathrm{sto}}^2(1-3c_{\mathrm{inv}}^2\tau^{-1}) - (2+\frac{3}{2c_{\mathrm{eq}}^2(1+\beta_2)}\theta_{\mathrm{sto}}^2c_{\mathrm{inv}}^2)\varepsilon_\vartheta}{1+\varepsilon_\vartheta-\theta_{\mathrm{sto}}^2}
  \end{equation*}
  and $\alpha < \frac{1}{2}\varepsilon_\vartheta\omega_{\ell}\tau c_{\mathrm{rel}}^{-2}c_{\mathrm{eq}}^{-2}$, which leads to
  \begin{equation*}
    c_4
    < \frac{1}{2}\varepsilon_\vartheta + \frac{1}{2}\varepsilon_\vartheta + 1 + \varepsilon_\vartheta - 2\varepsilon_\vartheta
    = 1.
  \end{equation*}
  Note that the upper bound of $\varepsilon_\vartheta$ implies that the upper bound of $\varepsilon_{\mathrm{sto}}$ is positive. 
  Moreover, since $(1+\varepsilon_\vartheta - \theta_{\mathrm{sto}}^2) > 0$ and $1-3(1+\varepsilon_{\mathrm{det}})c_{\mathrm{inv}}^2\tau^{-1} < 1$ for any $\tau>0$, it follows
  \begin{equation*}
    0 
    < (1+\varepsilon_{\mathrm{sto}})(1+\varepsilon_\vartheta) - \theta_{\mathrm{sto}}^2\bigl(1+\varepsilon_{\mathrm{sto}} - 3(1+\varepsilon_{\mathrm{det}})c_{\mathrm{inv}}^2\tau^{-1}\bigr)
  \end{equation*}
  and thus $0 < c_4$.
  Finally, 
  \begin{equation*}
    \frac{1}{2\varepsilon_{\mathrm{det}}} < \beta_2 < \frac{1}{\varepsilon_{\mathrm{det}}}
    \qquad\mbox{and}\qquad
    \alpha<\frac{\omega_{\ell}}{2c_{\mathrm{rel}}^{2}}
  \end{equation*}
  lead to $0 < c_3 < 1$.
  Choosing $\alpha,\varepsilon_{\mathrm{det}}$ and $\varepsilon_{\mathrm{sto}}$ smaller than the minimum of the respective bounds above yields $0<c_1,\dots,c_4<1$ and thus concludes the proof with $\delta_\ell := \max\{ 1-\alpha, c_1, \dots, c_4 \} < 1$.
\end{proof}

\begin{remark}[error reduction and convergence]%
  \label{rem:convergence:no_convergence}
  Theorem~\ref{thm:convergence:convergence} proves reduction of the quasi-error in a single iteration step.
  However, as it is possible that $\delta_\ell$ converges faster towards one then $\exp(-\ell^{-k})$ for $k>1$ as $\ell\to\infty$, this might not imply convergence of the quasi-error to zero.
  Furthermore, it is impossible to bound $\delta_\ell$ independently of $\ell$ for the lognormal diffusion coefficient~\eqref{eq:setting:a} since function spaces with adapted Gaussian measures have to be employed.
  As a consequence,~\eqref{eq:setting:boundedness_B} and~\eqref{eq:setting:coercivity_B} hold with respect to differently weighted norms.
  This causes a dependence of the Lipschitz constants in Themorem~\ref{thm:properties:lipschitz_det} and Theorem~\ref{thm:properties:lipschitz_sto} on the size of the active set $\Lambda_\ell$ and yields no positive lower or upper bound for $\Vert\bullet\Vert_B / \Vert\grad\bullet\Vert_{\pi_\rho,D}$~\cite[Lemma 2.32]{SchwabGittelson2011}.
\end{remark}

Remark~\ref{rem:convergence:no_convergence} also implies that $\delta_\ell$ can be bounded independently of $\ell$ if $a$ is bounded uniformly from above and below.
Hence, as a byproduct of Theorem~\ref{thm:convergence:convergence} we obtain a generalization of the convergence result in~\cite[Theorem 7.2]{EigelGittelson2014convergence} from affine to arbitrary uniformly bounded and positive diffusion coefficients.

\begin{corollary}[convergence for bounded coefficients]%
  \label{cor:convergence:convergence_bounded_a}
  Consider the setting of Theorem~\ref{thm:convergence:convergence} and additionally assume that the coefficient $a_N$ is uniformly positive and bounded, i.e., there exist $0 < \check{a} < \hat{a} < \infty$ such that $\check{a} \leq a_N(x,y) \leq \hat{a}$ for all $x\in D$ and almost all $y\in\Gamma$.
  Then there exist $0<\delta<1$, $\omega>0$ independent of $\ell$ and $\tau>0$, such that
  \begin{equation*}%
      \Vert u-u_{\ell+1}\Vert_{B}^2 + \omega \, \eta_{\mathrm{det},\ell+1}^2 + \omega\tau \, \eta_{\mathrm{sto},\ell+1}^2
      \leq \delta\Bigl( \Vert u-u_{\ell}\Vert_{B}^2 + \omega \, \eta_{\mathrm{det},\ell}^2 + \omega\tau \, \eta_{\mathrm{sto},\ell}^2\Bigr).
  \end{equation*}
\end{corollary}

\begin{proof}%
  Due to the boundedness of $a_N$ the problem is well posed in $L^2(\Gamma,\pi_0;\mathcal{X})$ and no adapted function spaces are required. 
  As a consequence $\Vert \bullet \Vert_B \approx \Vert \grad \bullet\Vert_{\pi_0,D}$ proves the claim.
\end{proof}

We close the section with a remark on the convergence in practical scenarios.

\begin{remark}[convergence in applications]%
  \label{rem:convergence:practical_convergence}
  The diffusion coefficient~\eqref{eq:setting:a} becomes uniformly positive and bounded on the subdomain $\Gamma_k = \{ y\in\Gamma \ \vert\ \vert y \vert \leq k\}\subset\Gamma$ for any $k > 0$.
  Moreover, the probability of $y\in\Gamma\setminus\Gamma_k$ vanishes as $k\to\infty$.
  As most practical applications rely on a finite number of realizations $y^{(i)}$ of $y$ to approximate stochastic integrals, we can assume all $y^{(i)}\in\Gamma_k$ for some $k>0$ with probability arbitrarily close to $1$.
  Corollary~\ref{cor:convergence:convergence_bounded_a} thus implies the convergence of Algorithm~\ref{alg:estimator:adaptive} for the unbounded diffusion coefficient~\eqref{eq:setting:a} with probability arbitrarily close to $1$ in practical applications.
\end{remark}
 
\section{Numerical experiments}%
\label{sec:experiments}

In this section we show that the quasi-error reduction of Algorithm~\ref{alg:estimator:adaptive} can also be observed in numerical experiments.
For that we rely on typical benchmark problems as used in for example~\cite{DolgovScheichl2019alsCross, EigelMarschall2020lognormal, EigelFarchmin2022avmc}.
As spatial domain we consider the L-shape $D=(0,1)^2\setminus[0.5,1]^2$.
The derived total error estimator $\eta$ is used to steer the adaptive refinement of the triangulation $\mathcal{T}$ and the space $\Lambda_d$ as described in Algorithm~\ref{alg:estimator:adaptive}.

To validate the reliability of the estimator and its contributions in the adaptive scheme, we compute an empirical approximation of the true $L^2(\Gamma, \pi;\mathcal{X})$-error using $N_{\mathrm{MC}}$ samples, i.e.
\begin{equation}%
  \label{eq:experiments:error}
  \Vert \grad(u - u_{\ell}) \Vert_{\pi_0,D}^2
  \approx \mathcal{E}(u_{\ell})^2
  = \frac{1}{N_{\mathrm{MC}}} \sum_{i=1}^{N_{\mathrm{MC}}} \Vert \grad \hat{u}(y^{(i)}) - \grad u_{\ell}(y^{(i)}) \Vert_{\pi_0, D}^2.
\end{equation}
Here, $\hat{u}(y^{(i)})$ is the deterministic sampled solution $u(y^{(i)})$ projected onto a uniform refinement $\hat{\mathcal{T}}$ of the finest FE mesh $\mathcal{T}_L$ obtained in the adaptive refinement loop.
Since all triangulations generated by Algorithm~\ref{alg:estimator:adaptive} as well as $\hat{\mathcal{T}}$ are nested, we employ simple nodal interpolation of each $u_{\ell}$ onto $\hat{\mathcal{T}}$ to guarantee $u_{\ell}\in\mathcal{V}_N(\Lambda_d; \hat{\mathcal{T}},p)$.
The choice of $N_{\mathrm{MC}}=250$ proved to be sufficient to obtain consistent estimates of the error in our experiments as well as in other works (cf.~\cite{EigelMarschall2020lognormal,EigelFarchmin2022avmc}).

As benchmark problem we consider the stationary diffusion problem~\eqref{eq:introduction:darcy} with constant right-hand side $f(x,y) = 1$.
We assume the coefficients of the affine diffusion field~\eqref{eq:setting:gamma} to enumerate planar Fourier modes in increasing total order, i.e.,
\begin{align*}
  \gamma_m(x) = \frac{9}{10\zeta(\sigma)} m^{-\sigma} \cos\bigl(2\pi\beta_1(m) x_1\bigr)\, \cos\bigl(2\pi\beta_2(m) x_2\bigr),
  \qquad m=1,\dots,\hat M,
\end{align*}
where $\zeta$ is the Riemann zeta function and, for $k(m) = \lfloor -\frac{1}{2} + \sqrt{\frac{1}{4} +2m} \rfloor$, $\beta_1(m) = m - k(m)(k(m)+1)/2$ and $\beta_2(m) = k(m) - \beta_1(m)$.
For our experiments we consider an expansion length of $\hat M=20$, decay $\sigma=2$, choose $\rho=1$ and $\vartheta=0.1$ similar to~\cite{EigelMarschall2020lognormal} and discretize~~\eqref{eq:setting:a} in the same finite element space as the solution, i.e., conforming Lagrange elements of order $p=1$ or $p=3$.
All finite element computations are conducted with the \texttt{FEniCS} package~\cite{fenics}.
For the stochastic discretization we rely on a low-rank tensor decomposition, i.e., the Tensor Train format~\cite{Oseledets2011}, to approximate all stochastic quantities.
In particular we build on the same framework as~\cite{EigelFarchmin2022avmc}, which uses the open source software package \texttt{xerus}~\cite{xerus}.

The constant right-hand side has an exact representation in the Tensor Train format, see e.g.~\cite{EigelFarchmin2022avmc} for the construction.
To assure that the approximation $a_N$ of the lognormal diffusion coefficient~\eqref{eq:setting:a} is sufficient, we employ the approach described in~\cite{EigelFarchmin2021expTT}.
In particular we enforce that the relative approximation error $\Vert a-a_N\Vert_{L^2(\Gamma,\pi_0;L^\infty(D))}$ is at least one order of magnitude smaller then the empirical error~\eqref{eq:experiments:error}.  %

Algorithm~\ref{alg:estimator:adaptive} is instantiated with a single mode $M=1$ discretized with an affine polynomial, i.e., dimension $d_1=2\in\mathbb{N}^1$.
The initial spatial mesh consists of $\vert\mathcal{T}_1\vert=143$ triangles for affine and $\vert\mathcal{T}_1\vert=64$ for cubic ansatz functions.
The marking parameters are set to $\theta_{\mathrm{det}}=0.3$ and $\theta_{\mathrm{sto}}=0.5$, respectively.
To achieve equilibration of the two estimator contributions we choose $c_{\mathrm{eq}}=5$.
We terminate Algorithm~\ref{alg:estimator:adaptive} after $L=12$ iteration steps.

\begin{figure}[htp]
  \begin{center}
    \includegraphics[height=5.5cm]{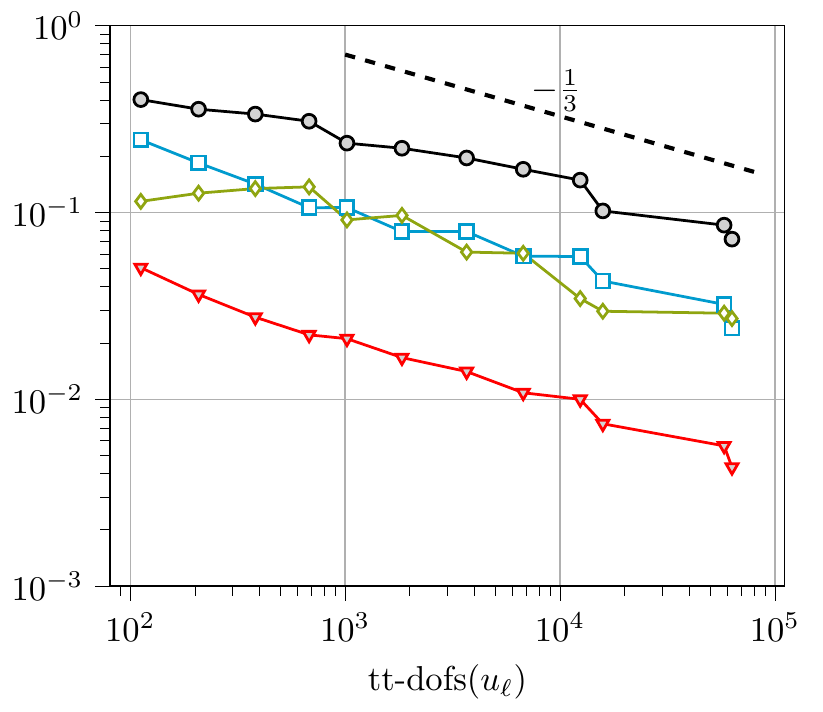}
    \includegraphics[height=5.5cm]{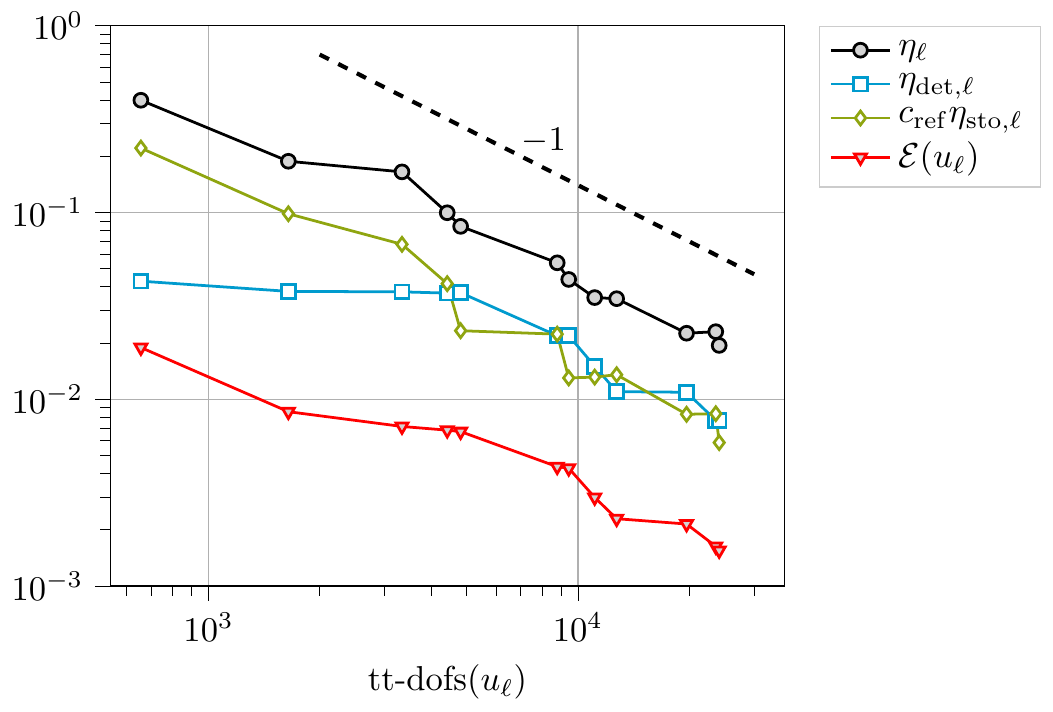}
  \end{center}
  \caption{%
    Reduction of estimator and error on the L-shaped domain for affine (left) and cubic (right) Lagrange finite elements with respect to the Tensor Train degrees of freedom of the Galerkin projection $u_N$.
  }%
  \label{fig:experiments:convergence}
\end{figure}

Figure~\ref{fig:experiments:convergence} depict the sampled root mean squared $H_0^1(D)$ error $\mathcal{E}(u_{\ell})$, the overall error estimator $\eta(u_{\ell})$ and the two estimator contributions $\eta_{\mathrm{det}}(u_\ell)$ and $\eta_{\mathrm{sto}}(u_\ell)$ for affine and cubic Lagrange finite elements, respectively.
The plots depict error and estimator against the degrees of freedom (dofs) of the coefficient tensor of the Galerkin projection $u_N$ compressed by the Tensor Train format, i.e.
\begin{equation*}
  \operatorname{tt-dofs}(u_N) = Jr_1 - r_1^2 + \sum_{m=1}^{M-1} (r_m d_m r_{m+1} - r_{m+1}^2) + r_M d_M,
\end{equation*}
where $r = (r_1,\dots,r_M)\in\mathbb{N}^M$ are the Tensor Train ranks, see e.g.~\cite{Holtz2012} for details.

The estimator mirrors the behaviour of the error with a consistent overestimation by a factor $c_{\mathrm{rel}}\approx10$, which is in line with Theorem~\ref{thm:estimator:reliability}.
Additionally, the deterministic estimator contribution $\eta_{\mathrm{det}}$ captures the singularity of the L-shaped domain and prioritizes to refine the mesh at the reentrant corner as known from deterministic adaptive FE methods, which is in line with previous results~\cite{EGSZ14,BespalovPraetorius2019convergence,Bespalov2018,Bespalov2021twolevel,EMPS20}.
We also observe that Algorithm~\ref{alg:estimator:adaptive} focusses on refinement of the finite element mesh for $p=1$ and tends to enlarge the stochastic space in the case $p=3$.
Again, this is in line with the expectations, as the higher regularity of cubic finite elements allows for coarser spatial resolution.

Finally we note that the experiments are in line with the results of Theorem~\ref{thm:convergence:convergence} as we observe a reduction of both error and estimator in each iteration.
Interestingly, we even see that the algorithm reduces both error and estimator with an overall constant rate, which is consistent with the results of e.g.~\cite{EigelMarschall2020lognormal,EigelFarchmin2022avmc}.
This is a stronger behavior than predicted by Theorem~\ref{thm:convergence:convergence} and can be seen as a validation of Remark~\ref{rem:convergence:practical_convergence}, i.e., the diffusion coefficient $a$ is effectively only considered on a bounded domain by Algorithm~\ref{alg:estimator:adaptive}, which yields uniform boundedness of $a$ from above and below and thus Corollary~\ref{cor:convergence:convergence_bounded_a} can be applied. 
\section*{Acknowledgements}
M.\ Eigel acknowledges the partial support of the DFG SPP 1886 ``Polymorphic Uncertainty Modelling for the Numerical Design of Structures'' and the EMPIR project 20IND04-ATMOC.
This project (20IND04 ATMOC) has received funding from the EMPIR programme co-financed by the Participating States and from the European Union’s Horizon 2020 research and innovation programme.
N.\ Hegemann has received funding from the Federal Ministry for Economic Affairs and Climate Action (BMWK) in the frame of the "QI-Digital" initiative programme “Metrology for Artificial Intelligence in Medicine” (M4AIM) and the EPM project 22HLT01 QUMPHY. 
This project (22HLT01 QUMPHY) has received funding from the European Partnership on Metrology, co-financed from the European Union’s Horizon Europe Research and Innovation Programme and by the Participating States.

\printbibliography

\end{document}